\documentclass[12pt,reqno] {amsart}%
\usepackage{amsfonts}
\usepackage{amsmath}
\usepackage{amssymb}
\usepackage{xypic,color}
\usepackage{graphicx}%
 \theoremstyle{plain}

\newtheorem{problem}{Problem}

\newtheorem{theorem}{Theorem}[section]
\newtheorem{conjecture}{Conjecture}
\newtheorem{corollary}[theorem]{Corollary}
\newtheorem{definition}[theorem]{Definition}
\newtheorem{lemma}[theorem]{Lemma}
\newtheorem{proposition}[theorem]{Proposition}
\newtheorem{remark}[theorem]{Remark}
\newtheorem{example}[theorem]{Example}

\def\bn{\begin{definition}}
\def\en{\end{definition}}
\def\ba{\begin{array}}
\def\ea{\end{array}}
\def\be{\begin{equation}}
\def\ee{\end{equation}}
\def\bd{\begin{description}}
\def\ed{\end{description}}
\def\bu{\begin{enumerate}}
\def\eu{\end{enumerate}}
\def\bi{\begin{itemize}}
\def\ei{\end{itemize}}

\newcommand{\A}{\mathcal A}

\newcommand{\B}{\mathcal B}

\newcommand{\bea}{\begin{eqnarray}}
\newcommand{\eea}{\end{eqnarray}}
\newcommand{\bes}{\begin{equation*}}
\newcommand{\ees}{\end{equation*}}
\newcommand{\beas}{\begin{eqnarray*}}
\newcommand{\eeas}{\end{eqnarray*}}

\def\C{\mathbb{C}}

\def\N{\mathbb{N}}

\def\cA{\mathcal{A}}

\def\ds{\displaystyle}

\def\<{\langle}
\def\>{\rangle}
\allowdisplaybreaks
\usepackage[colorlinks,
            linkcolor=black,
            anchorcolor=black,
            citecolor=black,
            ]{hyperref}
\begin{document}
\title[]{Disjointness of M\"{o}bius from asymptotically periodic functions}

\author[]{Fei Wei}

\address{Yau Mathematical Sciences Center, Tsinghua University, Beijing 100084, China}

\email{weif@mail.tsinghua.edu.cn}

\maketitle

\begin{abstract}
We investigate Sarnak's M\"obius Disjointness Conjecture through asymptotically periodic functions. It is shown that Sarnak's conjecture for rigid dynamical systems is equivalent to the disjointness of M\"obius from asymptotically periodic functions. We give sufficient conditions and a partial answer to the later one. As an application, we show that Sarnak's  conjecture holds for a class of rigid dynamical systems, which improves an earlier result of Kanigowski-Lema{\'{n}}czyk-Radziwi{\l}{\l}.
\textsc{MSC classes}. 37A55, 11N37
\end{abstract}
\section{Introduction}
Let $\N=\{0,1,2,\ldots\}$ denote the set of natural numbers and $\N^*=\{1,2,\ldots\}$. Functions from $\N$ (or $\N^*$) into $\C$ are called \emph{arithmetic functions}. Many problems in number theory can often be reformulated in terms of properties of arithmetic functions. For example, the M\"obius function $\mu(n)$ is defined by 0 if $n$ is not square free (i.e., divisible by a nontrivial square), and $(-1)^{r}$ if $n$ is the product of $r$ distinct primes.
It is well known that the Prime Number Theorem is
equivalent to that $\sum_{n\leq x}\mu(n)=o(x)$; the Riemann Hypothesis holds if and only if $\sum_{n\leq x}\mu(n)=o(x^{\frac{1}{2}+\epsilon})$,
for any $\epsilon>0$.

An arithmetic function $f$ is said to be \emph{disjoint} from another one $g$ if $\sum_{n=1}^{N}f(n)
\overline{g}(n)=o(N)$. Disjointness is a commonly concerned relation between arithmetic functions.
The disjointness of M\"obius from arithmetic functions plays an important role in number theory since they reflect certain random distribution among the values of the M\"obius function and are closely related to the distribution of primes. For example, the disjointness of $\mu(n)$ from periodic functions is equivalent to the prime number theorem in arithmetic progressions. Sarnak  (\cite{Sar}) conjectured that the M\"obius function is disjoint from all arithmetic functions arising from
any topological dynamical systems with zero topological entropy. More specifically,
\begin{conjecture}[Sarnak's M\"{o}bius Disjointness Conjecture (SMDC)]
Let $X$ be a compact Hausdorff space and $T$ a continuous map on $X$ with zero topological entropy, then
\[\lim_{N\rightarrow \infty}\frac{1}{N}\sum_{n=1}^{N}\mu(n)f(T^{n}x_{0})=0\]
for any $x_{0}\in X$ and $f\in C(X)$.
\end{conjecture}

In recent years, a lot of progress have been made on Conjecture 1. See \cite{JB,JPT,FJ,FH,GT,GW,HWY,Kan,JP,LOZ,Pec,Vee,Wang,W,xu}, to list a few. In the following, we shall discuss only the results that are more related to this paper. Sarnak proved that SMDC is implied by Chowla's conjecture which is stated as follows \cite{Ch}.

\begin{conjecture}[Chowla's conjecture]
Let $a_{0},a_{1},a_{2},\ldots,a_{m}$ be distinct natural numbers, and $i_{s}\in \{1,2\}$ for $s=0,1,2,\ldots,m$, not all  $i_{s}$ are even numbers. Then
\[\lim_{N\rightarrow \infty}\frac{1}{N}\sum_{n=1}^{N}\mu^{i_{0}}(n+a_{0})\mu^{i_{1}}(n+a_{1})\cdot\cdot\cdot\mu^{i_{m}}(n+a_{m})=0.\]
\end{conjecture}
Chowla's conjecture is a longstanding open problem in number theory. It is open even in one of its simplest forms: $\sum_{n=1}^{N}\mu(n)\mu(n+2)=o(N)$. This estimate should be closely related to the twin prime conjecture.

\subsection{Asymptotically periodic functions} In order to use tools from operator algebra to study Sarnak's conjecture, Ge introduced the following notion of asymptotically periodic function in the survey paper \cite{LG} and proved that the  M\"{o}bius function is disjoint from certain asymptotically periodic functions.
\begin{definition} \label{def:asymptotically periodic}
A function $f\in l^{\infty}(\mathbb{N})$ is called \textbf{asymptotically periodic}\footnote{This definition is a little weaker than \cite[Definition 5.7]{LG}, in which the sequence $\{n_{j}\}_{j=1}^{\infty}$ is independent of $E$.}
if for any mean state $E$, there is a sequence $\{n_{j}\}_{j=1}^{\infty}$ of positive integers such that $f-A^{n_{j}}f$ has limit zero in $\mathcal{H}_{E}$.
\end{definition}
In the above definition, the action $A$ on $l^{\infty}(\mathbb{N})$, the algebra of all bounded arithmetic functions endowed with the pointwise addition and multiplication, is defined as
\begin{equation}\label{20110114formula}
Af(n)=f(n+1),
\end{equation}
for all $f\in l^{\infty}(\mathbb{N})$ and $n\in \mathbb{N}$. The mean states $E$ on $l^{\infty}(\mathbb{N})$ are given by certain limits of $\frac{1}{N}\sum_{n=0}^{N-1}f(n)$ along ``ultrafilters" and $\mathcal{H}_{E}$ is the Hilbert space obtained by the GNS construction on $l^\infty(\N)$ with respect
to $E$. We refer readers to Section \ref{invariant means} for more details.

In this paper, we further study properties of asymptotically periodic functions and the M\"obius disjointness of asymptotically periodic functions, and give
some applications of these results to Sarnak's conjecture. We first introduce the following subclass of asymptotically periodic functions.
\begin{definition} \label{def:strongly asymptotically periodic}
A function $f\in l^{\infty}(\mathbb{N})$ is called \textbf{strongly asymptotically periodic} if for any mean state $E$, there is a sequence $\{n_{j}\}_{j=1}^{\infty}$ of positive integers such that
when $j$ goes to infinity, $f-A^{ln_{j}}f$ converges to zero in $\mathcal{H}_{E}$ uniformly with respect to all $l\in\N$.
\end{definition}
Here are some examples. The function $e^{2\pi i\sqrt{n}}$ is a strongly asymptotically periodic function. For any strictly increasing sequence $\{N_{j}\}_{j=0}^{\infty}$ ($N_{0}=0$) and any bounded sequence $\{a_{j}\}_{j=0}^{\infty}$ of complex numbers, define $f(n)=a_{j}$ when $N_{j}\leq n<N_{j+1}$. Then $f$ is strongly asymptotically periodic. If $\theta$ is an irrational number, then $e^{2\pi in\theta}$ is an asymptotically periodic function but not in the strong sense. The function $e^{2\pi i n^2\theta}$ with $\theta$ irrational is disjoint from all asymptotically periodic functions. These results and more examples of strongly asymptotically periodic functions are shown in Section \ref{invariant means}.

In \cite{Ebe}, Eberlein introduced the notion of weakly almost periodic (WAP) functions. These functions have been studied in dynamical systems (see e.g., \cite{GM,Vee2}). Moreover, all these functions can be realized in topological dynamical systems with zero topological entropy (see \cite[Theorem 9.1]{Vee}). We shall show that WAP functions belong to the class of asymptotically periodic functions satisfying conditions (\ref{restriction3/1}) and (\ref{asymptotical periodicity3/1}) below, see Proposition \ref{WAP is asymptotically periodic function}.

Interestingly, there are strongly asymptotically periodic functions that cannot be realized in any topological dynamical system with zero topological entropy. In the following we give an example to illustrate it.
\begin{proposition}\label{non zero entropy example}
Suppose $s\geq 1$ and $m_{1},\ldots,m_{s}\in \mathbb{N}$ with at least one $m_{i}\geq 1$. Let $f=A^{m_{1}}(\mu^2)\cdots A^{m_{s}}(\mu^2)$. Then $f(n)$ is a strongly asymptotically periodic functions. Moreover, $f(n)$ cannot be realized in any topological dynamical system with zero topological entropy, i.e., there does not exist a topological dynamical system $(X,T)$ such that the topological entropy of $T$ is zero and $f(n)=F(T^{n}x_{0})$ for some $F\in C(X)$ and $x_{0}\in X$.
\end{proposition}
\subsection{The M\"{o}bius disjointness of asymptotically periodic functions} We are interested in the following problem.
\begin{problem}\label{disjointness of mobius function with all asymptotically periodic functions}
Is $\mu$ disjoint from all asymptotically periodic functions?
\end{problem}
We now explain a motivation for us to investigate the above problem. The positive answer to Chowla's conjecture implies that the set $\{A^{n}\mu: n\geq 0\}$ is an orthogonal set
(with respect to a given mean state) of vectors of the same norm.
Denote the norm of $\mu$ as $c$. By Bessel's inequality, we have
\begin{equation} \label{eq}
\langle f,f\rangle\geq \frac1{c^2}\sum_{n\in\mathbb{N}}|\langle A^{n}\mu,f\rangle|^2
\end{equation}
for any $f\in l^\infty(\mathbb{N})$. Assume that $f$ is an asymptotically periodic function, then there is a sequence $\{n_j\}_{j=1}^{\infty}$ of distinct positive integers such that $\lim_{j\rightarrow \infty}|\langle A^{n_j}\mu,f-A^{n_{j}}f\rangle|=0$. This implies that $\lim_{j\rightarrow \infty}|\langle A^{n_j}\mu,f\rangle|=|\langle\mu,f\rangle|$,
then $\langle\mu,f\rangle=0$ by the inequality (\ref{eq}).

The process of exploring Problem \ref{disjointness of mobius function with all asymptotically periodic functions} motivates us to study the average value of the M\"{o}bius function in short arithmetic progressions. Precisely, we should estimate the second moment of this average: $\sum_{n=1}^{N}|\sum_{l=1}^{h}\mu(n+kl)|^2$. About this, we show the following result.
\begin{theorem}\label{main estimate}
Let $k$ be a positive integer. Then for any $h\geq 3$,
\begin{equation}\label{has relation with k}
\limsup_{N\rightarrow \infty}\frac{1}{N}\sum_{n=1}^{N}\Bigg|\sum_{l=1}^{h}\mu(n+kl)\Bigg|^2\ll \frac{k}{\varphi(k)}\frac{\log\log h}{\log h}h^2.
\end{equation}
\end{theorem}
Throughout this paper, $f\ll g$ means that there is an absolute constant $c$, such that $|f|\leq c|g|$;
$f=g+O(h)$ means $f-g\ll h$. We use $\varphi(k)$ to denote the Euler totient function.

We are more concerned about whether the left hand side of formula (\ref{has relation with k}) is still $o(h^2)$ when $k$ is far large than $h$. The estimate presented in formula (\ref{has relation with k}) implies that this holds for $k$ as large as $\exp(h^{o(1)})$ since $k/\varphi(k)\ll \log\log k$. We expect that the right hand side of formula (\ref{has relation with k}) is $o(h^2)$ independent of $k\ge 1$. This is likely to be true because the positive answer to Chowla's conjecture implies that the left hand side of formula (\ref{has relation with k}) should be $\frac{6}{\pi^2}h$.

In  Theorem \ref{main estimate}, we can replace $\mu$ by non-pretentious 1-bounded multiplicative functions such as the Liouville functions and $\mu(n)\chi(n)$, where $\chi$ is a Dirichlet character, see Lemma \ref{main theorem2}. Moreover, we recently extended Theorem \ref{main estimate} to the case that $\mu(n)$ is replaced by $\mu(n)e(P(n))$ for any $P(x)\in \mathbb{R}[x]$ (\cite{W}), and this is possible to be generalized to $\mu(n)$ twisted by any nilsequence.

Using Theorem \ref{main estimate}, we give a partial answer to Problem \ref{disjointness of mobius function with all asymptotically periodic functions}, which states that $\mu(n)$ is disjoint from a class of asymptotically periodic functions. Precisely,
\begin{theorem}\label{disjointfromasymptoticallyperiodicfunctions3/1}
Let $f\in l^{\infty}(\mathbb{N})$ satisfying that for any mean state $E$, there are sequences $\{h_{j}\}_{j=1}^{\infty}$ and $\{n_{j}\}_{j=1}^{\infty}$ of positive integers with
\begin{equation}\label{restriction3/1}
\lim_{j\rightarrow \infty}\frac{\log\log h_{j}}{\log h_{j}}\frac{n_{j}}{\varphi(n_{j})}=0
\end{equation}
such that
\begin{equation}\label{asymptotical periodicity3/1}
\lim_{j\rightarrow \infty}\frac{1}{h_{j}}\sum_{l=1}^{h_{j}}E(|f-A^{ln_{j}}f|^2)=0.
\end{equation}
Then we have
\[\lim_{N\rightarrow \infty}\frac{1}{N}\sum_{n=1}^{N}\mu(n)f(n)=0.\]
\end{theorem}
By the definition of strongly asymptotically periodicity, it is not hard to check the following result as an application of Theorem \ref{disjointfromasymptoticallyperiodicfunctions3/1}.
\begin{corollary}
Problem \ref{disjointness of mobius function with all asymptotically periodic functions} holds for all strongly asymptotically periodic functions.
\end{corollary}
For solving Problem \ref{disjointness of mobius function with all asymptotically periodic functions} completely, we provide a sufficient condition as follows.
\begin{proposition}\label{investigation of Problem 1}
Assume that
\begin{equation}\label{condition independent of k}
\limsup_{N\rightarrow \infty}\frac{1}{N}\sum_{n=1}^{N}\Bigg|\sum_{l=1}^{h}\mu(n+kl)\Bigg|^2=o(h^2),
\end{equation}
where the little ``$o$'' term is independent of $k\ge 1$.
Then Problem \ref{disjointness of mobius function with all asymptotically periodic functions} holds.
\end{proposition}

It is unknown if the converse of the above proposition is true. It is proved in Proposition \ref{the tested property of moebius} that the disjointness of M\"obius from all strongly asymptotically periodic functions is equivalent to that for any given $k$,
\[\limsup_{N\rightarrow \infty}\frac{1}{N}\sum_{n=1}^{N}\Bigg|\sum_{l=1}^{h}\mu(n+kl)\Bigg|^2=o(h^2).\]

Although many asymptotically periodic functions cannot be realized in topological dynamical systems with zero topological entropy, they can be approximated measure-theoretically by realizable functions (see Theorem \ref{zero entropy function approach}). This leads to the following result.
\begin{theorem}\label{main theorem 2}
Assume that Sarnak's M\"{o}bius Disjointness Conjecture is true, then Problem \ref{disjointness of mobius function with all asymptotically periodic functions} holds.
\end{theorem}
\subsection{Applications to Sarnak's conjecture for rigid dynamical systems}
Before introducing more results, we first recall the definition of rigid dynamical system. Let $(X,\mathcal{B},\nu,T)$ be a measure-preserving dynamical system, i.e., $X$ is a compact metric space, $T$ a continuous map on $X$, $\mathcal{B}$ the Borel $\sigma$-algebra of subsets of $X$ and $\nu$ a $T$-invariant Borel probability measure. Such a dynamical system is called \emph{rigid} if there is a sequence $\{n_{j}\}_{j=1}^{\infty}$ of positive integers such that for any $f\in L^2(X,\nu)$,
\[\lim_{j\rightarrow \infty}\|f\circ T^{n_{j}}-f\|_{L^2(\nu)}^2=0.\]
Rigid dynamical systems contain dynamical systems with discrete spectrum and a large class of skew products on the torus over a rotation of the circle \cite{Anzai2}. In the following for simplicity,
we use $(X,\nu,T)$ to denote a measure-preserving dynamical system.

From the viewpoint of dynamical systems, asymptotically periodic functions correspond to rigid measure-preserving dynamical systems (see Theorem \ref{A is rigid}). The major tool we use to build this connection between arithmetics and dynamics is anqie (of natural numbers), which was introduced by Ge in \cite{LG}. We refer readers to Section \ref{section:anqies} for knowledge on anqie. Based on this connection, corresponding to Problem \ref{disjointness of mobius function with all asymptotically periodic functions}, it is natural to consider the following problem.
\begin{problem}[Sarnak's conjecture for rigid dynamical systems]\label{rigid problem}
{\rm Let $X$ be a compact metric space and $T$ a continuous map on $X$. Suppose that $x_{0}\in X$ satisfies the following condition: for any $\nu$ in the weak* closure of $\{\frac{1}{N}\sum_{n=0}^{N-1}\delta_{T^{n}x_{0}}: N=1,2,\ldots\}$ in the space of Borel probability measures on $X$, there is a dense set $\mathcal{F}\subseteq C(X)$, such that for each $g(x)\in\mathcal{F}$ we can find a sequence $\{n_{j}\}_{j=1}^{\infty}$ (may depend on $\nu$, $g$) of positive integers satisfying
\begin{equation}\label{GAPR}
\lim_{j\rightarrow\infty}\|g\circ T^{n_{j}}-g\|_{L^2(\nu)}^2=0.
\end{equation}
Is it true that for any $f\in C(X)$,
\[\lim_{N\rightarrow \infty}\frac{1}{N}\sum_{n=1}^{N}\mu(n)f(T^{n}x_{0})=0?\]}
\end{problem}

\begin{proposition}\label{investigation of Problem 2}
Problem \ref{disjointness of mobius function with all asymptotically periodic functions} holds if and only if Problem \ref{rigid problem} holds.
\end{proposition}
When $\nu$ satisfies the condition in Problem \ref{rigid problem}, $(X,T,\nu)$ is rigid and then has zero measure-theoretic entropy (see e.g., \cite[Example 5.3.3]{KP7}), while $(X ,T)$ may not have zero topological entropy, see the paragraphs below Proposition \ref{asymptotically periodic has zeo metric entropy} for an example. Recently, in \cite{Kan}, Kanigowski, Lema{\'{n}}czyk and Radziwi{\l}{\l} gave a partial answer to Problem \ref{rigid problem}.
\begin{theorem}\cite[Theorem 2.1]{Kan}\label{polynomial rigidity}
With the same assumptions as Problem \ref{rigid problem}, if $T$ is a homeomorphism and for each $g(x)\in\mathcal{F}$ we can find a sequence $\{n_{j}\}_{j=1}^{\infty}$ (may depend on $\nu$, $g$) of positive integers satisfying either\\
{\rm(BPV rigidity)}: there is a constant $c>0$ such that $\sum_{p|n_{j}}\frac{1}{p}<c$ for any $j=1,2,\ldots$, and
\[\lim_{j\rightarrow\infty}\|g\circ T^{n_{j}}-g\|_{L^2(\nu)}^2=0.\]
or\\
{\rm(PR rigidity)}: for some $\delta>0$, the following holds:
\[\lim_{j\rightarrow\infty}\sum_{l=-n_{j}^{\delta}}^{n_{j}^{\delta}}\|g\circ T^{ln_{j}}-g\|_{L^2(\nu)}^2=0.\]
Then for any $f\in C(X)$,
\[\lim_{N\rightarrow \infty}\frac{1}{N}\sum_{n=1}^{N}\mu(n)f(T^{n}x_{0})=0.\]
\end{theorem}
Employing the estimate we obtained in Theorem \ref{main estimate}, we improve the above theorem to the following.
\begin{theorem}\label{comparison result}
With the same assumptions as Problem \ref{rigid problem}, if $T$ is a continuous map and for each $g(x)\in\mathcal{F}$, there are sequences $\{h_{j}\}_{j=1}^{\infty}$ and $\{n_{j}\}_{j=1}^{\infty}$ of positive integers with
\begin{equation}\label{restriction1}
\lim_{j\rightarrow \infty}\frac{\log\log h_{j}}{\log h_{j}}\frac{n_{j}}{\varphi(n_{j})}=0
\end{equation}
satisfying
\begin{equation}\label{asymptotical periodicity1}
\lim_{j\rightarrow \infty}\frac{1}{h_{j}}\sum_{l=1}^{h_{j}}\|g\circ T^{ln_{j}}-g\|_{L^2(\nu)}^2=0.
\end{equation}
Then for any $f\in C(X)$,
\[
\lim_{N\rightarrow \infty}\frac{1}{N}\sum_{n=1}^{N}\mu(n)f(T^{n}x_{0})=0.
\]
Moreover, the above disjointness also holds over short intervals in average, that is
\[\lim_{h\rightarrow \infty}\limsup_{N\rightarrow \infty}\frac{1}{Nh}\sum_{n=1}^{N}\Big|\sum_{l=1}^{h}\mu(n+l)f(T^{n+l}x_{0})\Big|=0.\]
\end{theorem}
In comparison with Theorem \ref{polynomial rigidity}, we relax $T$ to a continuous map. Also both BPV rigidity and PR rigidity are included in the scenario described by conditions (\ref{restriction1}), (\ref{asymptotical periodicity1}). See Remark \ref{extending conditions} for details.
There are examples that satisfy conditions (\ref{restriction1}), (\ref{asymptotical periodicity1}), but not BPV rigidity and PR rigidity (see Remark \ref{comparison rate}). Indeed, we show that conditions (\ref{restriction1}), (\ref{asymptotical periodicity1}) hold for any $(X,\nu,T)$ with discrete spectrum in Proposition \ref{discre spectrum}, while the set of these dynamical systems
are not strictly contained in the set of rigid dynamical systems satisfying
BPV rigidity or PR rigidity (see Remark \ref{extending conditions519}).

Related to the above result, we recently proved that Sarnak's conjecture holds for product flows between rigid dynamical systems satisfying conditions in Theorem \ref{comparison result} and affine linear flows on compact abelian groups of zero topological entropy \cite{W}.

Our paper is organized as follows.
In Section \ref{section:preliminaries}, we list some frequently used notation, and prove some preliminary results.
In Section \ref{section:anqies}, we study properties of anqies and
describe the topological characterizations of anqies in terms of the generating arithmetic functions.
We perform the GNS constructions on anqies, and show some examples of asymptotically and strongly asymptotically periodic functions in Section \ref{invariant means}, where Proposition \ref{non zero entropy example} is proved.
In Section
\ref{mean states}, we study the
connections between arithmetic functions and measure-preserving dynamical systems.
In Section \ref{estimate of second moment}, we show the estimate about the self-correlations of the M\"obius stated in Theorem \ref{main estimate}.
In Section \ref{proof of two theorems},
we prove Theorems \ref{disjointfromasymptoticallyperiodicfunctions3/1}, \ref{main theorem 2}, and Proposition \ref{investigation of Problem 1}.
As applications, we prove Proposition \ref{investigation of Problem 2} and Theorem \ref{comparison result} in Section \ref{applications}.

This work arose as part of my Ph.D. thesis at the Chinese Academy of Sciences \cite{FW} under the supervision of Professor Liming Ge. We refer to \cite{Co,KR} for basics and preliminary results in operator algebra, to \cite{Gl,IK} for that on topological dynamics and number theory.

{\bf Acknowledgements.}
I am indebted to my advisor Liming Ge for encouraging me in this research.
I am very grateful to Peter Sarnak  for insightful comments on the manuscript. I heartily thank Jinxin Xue for providing helpful comments and suggestions on the manuscript.
I would like to thank  Arthur Jaffe for his support, and Boqing Xue and Wei Yuan for valuable discussions.
%I also wish to thank the referees for a careful reading of the paper and for helpful comments.
This research was supported in part by the University of New Hampshire, by the AMSS of the Chinese Academy of Sciences, and by Grant TRT 0159 from the Templeton Religion Trust, and by the fellowship of China Postdoctoral Science Foundation 2020M670273.

\section{Preliminaries}\label{section:preliminaries}
In this section, we prove some preliminary results. First, we list some notation that will be used.

Let $\mathcal{H}$ be a Hilbert space. Denote by $\mathcal{B}(\mathcal{H})$ the algebra consists of all bounded linear operators on $\mathcal{H}$.
By Riesz representation theorem, for any $T\in\mathcal{B}(\mathcal{H})$,
there is a unique bounded linear operator $T^{*}$ satisfying $\langle Tx,y\rangle=\langle x,T^{*}y\rangle$ for any $x,y\in\mathcal{H}$.
Such a $T^*$ is called the \emph{adjoint of $T$}. We call a norm-closed *-subalgebra of $\mathcal{B}(\mathcal{H})$ a \emph{C*-algebra}. In this paper, we always assume that all C*-algebras are unital.

Suppose that $\mathcal{A}$ is a C*-algebra. We use $\mathcal{A}^{\sharp}$ to denote the set of all bounded linear functionals on $\mathcal{A}$. Denote by $(\mathcal{A}^{\sharp})_{1}$ the unit ball in $\mathcal{A}^{\sharp}$, i.e., $(\mathcal{A}^{\sharp})_{1}=\{\rho\in \mathcal{A}^{\sharp}:\|\rho\|\leq 1\}$. In general, the space $\mathcal{A}^{\sharp}$ can be equipped with many topological structures. Among them, the norm topology and weak* topology are used most frequently. For $\rho\in \mathcal{A}^{\sharp}$, its norm is given by $\|\rho\|=\sup_{x\in \mathcal{A},\|x\|\leq 1}|\rho(x)|$. When $x\in \mathcal{A}$, the equation $\sigma_{x}(\rho)=|\rho(x)|$ defines a semi-norm on $\mathcal{A}^{\sharp}$. The family $\{\sigma_{x}: x\in \mathcal{A}\}$ of semi-norms determines the \emph{weak* topology} on $\mathcal{A}^{\sharp}$. Note that each $\rho_{0}\in \mathcal{A}^{\sharp}$ has a base of neighborhoods consisting of sets of the form $\{\rho\in \mathcal{A}^{\sharp}: |\rho(x_{j})-\rho_{0}(x_{j})|<\epsilon\}$ ($j=1,\ldots,m$), where $\epsilon>0$ and $x_{1},\ldots,x_{m}\in \mathcal{A}$.

A non-zero linear functional $\rho$ on an abelian C*-algebra $\A$ is called \emph{a multiplicative state} if for any $A,B\in \mathcal{A}$, $\rho(AB)=\rho(A)\rho(B)$.

Suppose now that $\mathcal{A}$ is an abelian C*-algebra and $X$ is its maximal ideal space. We define the map $\gamma: \mathcal{A}\rightarrow C(X)$ by
\begin{equation}\label{Gelfand transform}
\gamma(f)(\rho)=\rho(f), \quad f\in \cA, \,\rho \in X.
\end{equation}
Here we use the fact that $X$ is also the space of all multiplicative states of $\mathcal{A}$. The map $\gamma$ is known as the \emph{Gelfand transform} from $\mathcal{A}$ onto $C(X)$, which is a *-isomorphism (see, e.g., \cite[Theorem 2.1]{Co}).

It is known that the above Hausdorff space $X$ is weak* compact. Next we show that $\mathcal{A}$ is countably generated as an abelian C*-algebra if and only if $X$ is metrizable and the topology induced by the metric coincides with the weak* topology. The sufficient part directly follows from \cite[Remark 3.4.15]{KR}. The necessary part is shown in the following proposition.

\begin{proposition}\label{metrizable}
Let $\mathcal{A}$ be an abelian C*-algebra. If $\mathcal{A}$ is countably generated, then $(\mathcal{A}^{\sharp})_{1}$ is metrizable and the toplology induced by the metric is equivalent to the weak* topology on $(\mathcal{A}^{\sharp})_{1}$. In particular,
the maximal ideal space of $\mathcal{A}$ is a compact metrizable space.
\end{proposition}
\begin{proof}
Since $\mathcal{A}$ is countably generated, there is a countable dense subset in $\mathcal{A}$.
Let $\{g_{1},g_{2},\ldots\}$ be a dense subset of $(\mathcal{A})_1$, the unit ball in $\mathcal{A}$. For any $\rho_{1},\rho_{2}\in (\mathcal{A}^{\sharp})_{1}$,
we define $d(\rho_{1},\rho_{2})=\sum_{i=1}^{\infty}\frac{|(\rho_{1}-\rho_{2})(g_{i})|}{2^{i}}$.
It is not hard to check that $d$ is a metric on $(\mathcal{A}^{\sharp})_{1}$.
Moreover, for any net $\{\rho_{\alpha}\}$ of elements of $(\mathcal{A}^{\sharp})_{1}$, the net $\{d(\rho_{\alpha},\rho)\}$ converges to $0$ is equivalent to the condition that,
for any $i\geq 1$, the net $\{\rho_{\alpha}(g_{i})\}$ converges to $\rho(g_{i})$.

Next, we show that the weak* topology is equivalent to the topology induced by the metric $d$ on $(\mathcal{A}^{\sharp})_{1}$.
Suppose that the net $\{\rho_{\alpha}\}$ of elements in $(\mathcal{A}^{\sharp})_{1}$, weak* converges to $\rho$. Then, for any $i\geq 1$,
the net $\{\rho_{\alpha}(g_{i})\}$ converges to $\rho(g_{i})$. Thus the net $\{d(\rho_{\alpha},\rho)\}$ converges to $0$.
Conversely, if the net $\{d(\rho_{\alpha},\rho)\}$ converges to $0$, where $\rho_{\alpha}\in (\mathcal{A}^{\sharp})_{1}$, then $\{\rho_{\alpha}(g_{i})\}$ converges to $\rho(g_{i})$ for any $i\geq 1$.
Note that, for any $\alpha$, $\|\rho_{\alpha}\|\leq 1$. Then for any $g\in \mathcal{A}$, the net
$\{\rho_{\alpha}(g)\}$ converges to $\rho(g)$. So the net $\{\rho_{\alpha}\}$ is weak* convergent to $\rho$ in $(\mathcal{A}^{\sharp})_{1}$.

By Alaoglu-Bourbaki theorem $(\mathcal{A}^{\sharp})_{1}$ is weak* compact. Let $X$ be the maximal ideal space of $\mathcal{A}$.
Then, relative to the weak* topology, $X$ is a closed subset of $(\mathcal{A}^{\sharp})_{1}$. From the above analysis, we see that the weak* topology on $(\mathcal{A}^{\sharp})_{1}$ coincides with the topology induced by the metric $d$ on it. Thus $X$ is a compact metrizable space.
\end{proof}

\begin{proposition}\label{the set of natural numbers are dense in X}
Suppose that $\mathcal{A}$ is a C*-subalgebra of $l^{\infty}(\mathbb{N})$ and $X$ the maximal ideal space of $\mathcal{A}$.
Let $\iota: \mathbb{N}\rightarrow X$ be the map given by
\begin{equation}\label{multiplicative state of point evaluation}
\iota(n):f \mapsto f(n),
\end{equation}
for any $f\in \mathcal{A}$. Then the weak* closure of $\iota(\mathbb{N})$ is $X$ $($write $\overline{\iota(\mathbb{N})} = X$$)$.
\end{proposition}
\begin{proof}
Assume on the contrary that $\overline{\iota(\mathbb{N})}\neq X$. Choose $y\in X \setminus \overline{\iota(\mathbb{N})}$. By Urysohn's lemma, there is a $G\in C(X)$
such that $G(y)=1$ and $G(x)=0$ for any $x\in \overline{\iota(\mathbb{N})}$. By equation (\ref{Gelfand transform}), for any $n\in \mathbb{N}$, $0=G(\iota(n))=\iota(n)(\gamma^{-1}G)=(\gamma^{-1}G)(n)$. Then $\gamma^{-1}(G)=0$ and $G=0$ correspondingly. This contradicts $G(y)=1$. Hence $\overline{\iota(\mathbb{N})}=X$.
\end{proof}

If $\iota$ is injective, then we can view $\N$ as a subset of $X$. For $\mathcal{A}=l^{\infty}(\mathbb{N})$, $\iota$ is injective. We shall use $\beta\mathbb{N}$ to denote the maximal ideal space of $l^{\infty}(\mathbb{N})$, which is also known as the Stone-\v{C}ech compactification of $\mathbb{N}$ \cite{E}.
Since $l^{\infty}(\mathbb{N})$ is not a separable C*-algebra, it is not countably generated as a C*-algebra.  Correspondingly,
the maximal ideal space $\beta\mathbb{N}$ is not metrizable.

\section{Anqie of $\mathbb{N}$}\label{section:anqies}
In this section, we briefly introduce the concept of anqie and list some results that will be used later. For more about anqie, we refer to \cite{LG} and \cite{Fei}.
\begin{definition}\label{definition of anqies}
Let $X$ be a compact Hausdorff space and $\iota$ a map from $\mathbb{N}$ to $X$ with dense range. We call $X$ \textbf{an anqie} $($of $\mathbb{N})$ if $\iota(n)\mapsto \iota(n+1)$ is a well-defined map on $\iota(\mathbb{N})$ and it can be extended to a continuous map from $X$ into itself. If we denote this extended map by $\sigma_{A}$, we also call $(X,\sigma_{A})$ \textbf{an anqie} $($of $\mathbb{N})$.
\end{definition}
We now explain a little about the above notion. For a general map $\iota: \mathbb{N}\rightarrow X$,  $\iota(n)\mapsto \iota(n+1)$ may not be well-defined, such as $\iota:\mathbb{N}\rightarrow S^{1}$ (the unit circle) defined as $\iota(n)=e^{2\pi i\sqrt{n}}$. Even though $\iota(n)\mapsto \iota(n+1)$ is well defined, it may not induce a continuous map on $X$, such as $\iota:\mathbb{N}\rightarrow S^{1}$ defined as $\iota(n)=e^{2\pi in^2\theta}$ with $\theta$ irrational. So an anqie of $\mathbb{N}$ preserves the addition structure of natural numbers when $\mathbb{N}$ is mapped to $X$.
Here is a simple example of anqie.
\begin{example}\label{irrational rotation}{\rm Let $\theta$ be an irrational number with
$0<\theta<1$. Define $\iota: n \mapsto e^{2\pi i n\theta}$, a map from $\N$
into $S^1$. It is easy to see that $\iota$ has a dense range in $S^1$ and $
e^{2\pi i n\theta}\mapsto e^{2\pi i(n+1)\theta}=e^{2\pi i\theta}e^{2\pi
in\theta}$ induces a
continuous map $z\mapsto e^{2\pi i\theta}z$, denoted by $\sigma_{A}$ on $S^1$. Thus $(S^1, \sigma_{A})$
is an anqie of
$\N$.}
\end{example}

Next we consider how to construct anqies of $\mathbb{N}$. One way to obtain anqies of $\mathbb{N}$ is to construct point transitive topological dynamical systems.
Recall that a topological dynamical system (or, equivalently an $\N$-dynamics) is a pair $(X,T)$, where $X$ is a compact Hausdorff space and
$T$ a continuous map on $X$. Suppose that $(X,T,x_{0})$ is a point transitive topological dynamical system, i.e., the set $\{T^{n}x_{0}: n\in \mathbb{N}\}$ is
dense in $X$. Then $\iota:n\mapsto T^{n}x_{0}$ is a map from $\mathbb{N}$ to $X$ with dense range. It is easy to see that $\iota(n)\mapsto \iota(n+1)$ can be extended to the continuous map $T$ on $X$. Then $(X,T)$ is an anqie of $\mathbb{N}$. Summarize the above analysis, we conclude that
an $\N$-dynamics $(X, T)$ is an anqie of $\N$ if it is point transitive. In this construction, the structure of anqie  depends on the choice of the transitive point.

Another method to construct anqies is through C*-algebras.
\begin{proposition}\label{anqie is a shif invariant algebra}
Suppose that $\mathcal{A}$ is a C*-subalgebra of $l^{\infty}(\mathbb{N})$ and $X$ the maximal ideal space of $\mathcal{A}$. Then $X$ is an anqie of $\mathbb{N}$ with the map $\iota$ given by equation (\ref{multiplicative state of point evaluation}) if and only if $\mathcal{A}$ is closed under the action $A$ defined in (\ref{20110114formula}), i.e., $Af\in \mathcal{A}$ for any $f\in \mathcal{A}$.
\end{proposition}

\begin{proof}
Suppose that $X$ is an anqie of $\mathbb{N}$. Then the map $\iota(n)\mapsto \iota(n+1)$ is extended to a continuous map on $X$, denoted by $\sigma_{A}$. Given $f\in \mathcal{A}$, assume that $F=\gamma(f)\in C(X)$ (see equation (\ref{Gelfand transform})). Note that $F\circ \sigma_{A}\in C(X)$. Let $g= \gamma^{-1}(F\circ \sigma_{A})$ in $\mathcal{A}$. Then $g(n)=F\circ \sigma_{A}(\iota(n))=F(\iota(n+1))=f(n+1)=Af(n)$. Thus $g=Af$ in $\mathcal{A}$. This shows that $\mathcal{A}$ is $A$-invariant.

On the other hand, suppose that $\mathcal{A}$ is closed under $A$. Let $\sigma_{A}$ be the map from $X$ to itself given by $\sigma_{A}\rho(f)=\rho(Af)$ for any $\rho\in X$ and $f\in \mathcal{A}$.
It is easy to see that $\sigma_{A}(\iota(n))=\iota(n+1)$.
Now we show that $\sigma_{A}$ is a continuous map on $X$. If $\{\rho_{\alpha}\}$ is a weak* convergent net of elements of $X$, with limit $\rho$, then
for any $f\in \A$, $\rho_{\alpha}(Af)=(\sigma_{A}\rho_{\alpha})(f)$ converges to $\rho(Af)=(\sigma_{A}\rho)(f)$.
Thus the net $\{\sigma_{A}\rho_{\alpha}\}$ weak* converges to $\sigma_{A}\rho$ in $X$. Hence $\sigma_{A}$ is the continuous map on $X$ extended by $\iota(n)\mapsto \iota(n+1)$ and $X$ is an anqie of $\mathbb{N}$.
\end{proof}
From the above proposition, we can obtain anqies of $\mathbb{N}$ through constructing $A$-invariant C*-subalgebras of $l^{\infty}(\N)$. In particular, we often consider the anqie generated by a single arithmetic function $f$, i.e., the C*-algebra generated by $\{1, A^jf: j\in \mathbb{N}\}$. Denote it by $\mathcal{A}_{f}$. We use $X_{f}$ to denote the maximal ideal space of $\mathcal{A}_{f}$. From Proposition \ref{anqie is a shif invariant algebra}, we know that $\iota(n)\mapsto \iota(n+1)$ can be extended to a continuous map on $X_{f}$, denoted by $\sigma_{A}$. We also call $(X_{f},\sigma_{A})$ or $\mathcal{A}_{f}$ \emph{the anqie generated by $f$}. Let $\overline{f(\N)}$ denote the closure of $f(\N)$ in the complex plane $\C$. Since $\A_f$ contains $f$, there is a continuous map from $X_f$ onto
$\overline{f(\N)}$. But these two spaces may not be the same.

The following theorem describes $X_{f}$ in terms of $\overline{f(\mathbb{N})}$ and gives a representation of $\sigma_{A}$
(corresponding to the Bernoulli shift on a product space).
\begin{theorem}\cite[Theorem 2.3]{LG}\label{the method to compute the anqie}
Suppose that $f$ is a function in $l^{\infty}(\mathbb{N})$ and that $X_{f}$ is the maximal ideal space of the anqie $\mathcal{A}_{f}$ generated by $f$.
Denote by $\prod_{\mathbb{N}}\overline{f(\mathbb{N})}$ the Cartesian product of $\overline{f(\mathbb{N})}$ indexed by $\mathbb{N}$, endowed with the product topology.
Assume that $B$ is the Bernoulli shift on $\prod_{\mathbb{N}}\overline{f(\mathbb{N})}$
defined by \[B:(a_{0},a_{1},a_{2},\ldots)\mapsto (a_{1},a_{2},a_{3},\ldots).\]
Let $F$ be the map from $X_{f}$ into $\prod_{\mathbb{N}}\overline{f(\mathbb{N})}$,
such that for any $\rho\in X_{f}$,
\begin{equation}
F(\rho)=(\rho(f),\rho(Af),\ldots).
\end{equation}
The following statements hold.

(\romannumeral1) The space $X_{f}$ is homeomorphic to $F(X_{f})$.

(\romannumeral2) $F(X_{f})$ is  the closure of $\{(f(n),f(n+1),\ldots):n\in \mathbb{N}\}$
in $\prod_{\mathbb{N}}\overline{f(\mathbb{N})}$.

(\romannumeral3) The restriction of the Bernoulli shift $B$ on $F(X_{f})$ is identified with $\sigma_{A}$ on $X_{f}$.
\end{theorem}

We refer readers to \cite[Theorem 3.4]{Fei} for a more general version of Theorem \ref{the method to compute the anqie}.
\begin{remark}\label{WAP}{\rm
It follows from the above theorem that for $f\in l^{\infty}(\mathbb{N})$, $X_{f}$ can be identified as the set of all pointwise limits of sequences $\{A^{n}f$, $n=0,1,2,\ldots\}$ in $l^{\infty}(\mathbb{N})$. This still holds when the semigroup $\mathbb{N}$ is replaced by the group $\mathbb{Z}$ or a general abelian topological group $G$. While for the case of $\mathbb{Z}$ or $G$, $X_{f}$ has been extensively studied in dynamical systems (see e.g., \cite{Ebe}, \cite{Vee2} and \cite{GM}). In \cite{Ebe}, Eberlein introduced the concept of weakly almost periodic function, i.e., $f\in l^{\infty}(G)$ with $X_{f}$ weak compact in $l^{\infty}(G)$. Let $W(G)$ denote the set of these functions. In \cite{Vee2}, Veech introduced a *-subalgebra $K(G)$ of $l^{\infty}(G)$ consisting of all $f\in l^{\infty}(G)$  with $X_{f}$ norm separable in $l^{\infty}(G)$, which contains $W(G)$. Recently, the M\"{o}bius disjointness of $W(\mathbb{Z})$ has been proved in \cite{Vee} and that of $K(\mathbb{Z})$ has been proved in \cite{HWY}.}
\end{remark}

Applying Theorem \ref{the method to compute the anqie}, we can obtain many interesting examples of anqies $(X_{f},\sigma_{A})$. 
The following two examples are given in \cite{LG, Fei}.

\begin{example}{\rm
Let $f(n)=e^{2\pi i\sqrt{n}}$, for $n\in \mathbb{N}$.
Then $X_{f}$ is homeomorphic to $\{e^{-\frac{1}{n}}f(n):n\in \mathbb{N}\}\cup S^{1}$, a subset of $\C$,
denoted by $X$.
And $\sigma_{A}$ is the identity map on $S^{1}$, while, on the set $\{e^{-\frac{1}{n}}f(n):n\in \mathbb{N}\}$,
$\sigma_{A}$ maps $e^{-\frac{1}{n}}f(n)$ to $e^{-\frac{1}{n+1}}f(n+1)$.}
\end{example}

\begin{example}\label{irrational rotation-3}{\rm
Suppose that $\theta$ is irrational and
$f(n)=e^{2\pi in^2\theta}$, then $X_{f}$ is homeomorphic to $S^{1}\times S^{1}$. Moreover,
if we identify $S^{1}\times S^{1}$ with $\mathbb{R}/\mathbb{Z}\times \mathbb{R}/\mathbb{Z}$, then we can rewrite the map $\sigma_{A}$ as
 \begin{align*}
        \sigma_{A}(\alpha_1, \alpha_2) =
        \begin{pmatrix}
            0 & 1 \\
            -1 & 2\\
        \end{pmatrix}
        \begin{pmatrix}
           \alpha_1 \\
           \alpha_2 \\
        \end{pmatrix} +
        \begin{pmatrix}
           0 \\
           2\theta
        \end{pmatrix}.
    \end{align*}
}\end{example}

At the end of this section, we state the following result which will be used in later parts of this paper. Recall that for two topological dynamical systems $(X_{1},T_{1})$ and $(X_{2},T_{2})$,
if there is a continuous surjective map $\varphi$ from $X_{1}$ onto $X_{2}$ such that $\varphi\circ T_{1}=T_{2}\circ \varphi$,
we call $\varphi$ \emph{a factor map} and $(X_{2},T_{2})$ a \emph{factor} of $(X_{1},T_{1})$.
Moreover, if $\varphi$ is a homeomorphism, we say that $(X_{1},T_{1})$ and $(X_{2},T_{2})$ are \emph{(topologically) conjugate (to each other)}.
\begin{proposition}\label{a factor}
Let $f$ be an arithmetic function realized in $(X,T)$, i.e., there is a continuous function $g\in C(X)$ and $x_{0}\in X$
such that $f(n)=g(T^n x_{0})$. Suppose that $X_{f}$ is the maximal ideal space of the anqie generated by $f$.
Let $Y$ be the closure of the set $\{T^{n}x_{0}:n\in \mathbb{N}\}$ in $X$. Then $(X_{f},\sigma_{A})$ is a factor of $(Y,T)$.
\end{proposition}
\begin{proof}
Since $\varrho: n\mapsto T^{n}x_{0}$ is a map from $\mathbb{N}$ to $Y$ with dense range,
it induces an embedding from $C(Y)$ into $l^{\infty}(\mathbb{N})$ (denoted by $\varrho$ again),
i.e., for any $h\in C(Y)$, $\varrho(h)(n)=h(T^{n}x_{0})$.
Then $\varrho (C(Y))$ is a C*-subalgebra of $l^{\infty}(\mathbb{N})$. Denote the maximal ideal space of $\varrho (C(Y))$ by $\widetilde{Y}$. By Proposition \ref{the set of natural numbers are dense in X}, the map $T^{n}x_{0}\mapsto \iota(n)$ can be extended to a homeomorphism from $Y$ onto $\widetilde{Y}$ (denoted by $\varrho$ again).
Note that $\varrho(g)=f$. So $\mathcal{A}_{f}$, the anqie generated by $f$, is a *-subalgebra of $\varrho (C(Y))$.

Since each multiplicative state on $\varrho (C(Y))$ (an element in $\widetilde{Y}$) is also a multiplicative state on $\mathcal{A}_{f}$
(an element in $X_{f}$) and that every maximal ideal in $\mathcal{A}_{f}$ extends to a maximal ideal (may not be unique) in $\varrho (C(Y))$,
the induced map $\pi$ from $\widetilde{Y}$ onto $X_{f}$ given by
\[
\pi(\rho)(h) = \rho(h),\quad \rho \in \widetilde{Y},\, h\in \cA_{f}.
\]
is continuous and surjective.
It is not hard to check that $\pi\circ\varrho:Y\rightarrow X_{f}$ is a factor map. Then $(X_{f},\sigma_{A})$ is a factor of $(Y,T)$.
\end{proof}
\section{Mean states and asymptotically periodic functions}\label{invariant means}
\label{Invariant means}
In number theory, we are often concerned with estimates
of the form $\frac{1}{x}\sum_{n\leq x}f(n)$. For this purpose,
we shall consider states on $l^{\infty}(\mathbb{N})$ given by certain limits of $\frac{1}{N}\sum_{n=0}^{N-1}f(n)$ along ``ultrafilters".
Then the inner product of two functions $f$ and $g$ given by the states is exactly certain limits of sums like $\frac{1}{N}\sum_{n=0}^{N-1}f(n)\overline{g(n)}$.

Recall that $\beta\mathbb{N}$ is the maximal ideal space of $l^{\infty}(\mathbb{N})$. Elements in $\beta\mathbb{N}\setminus \mathbb{N}$ are called \emph{free ultrafilters}. By Proposition \ref{the set of natural numbers are dense in X}, $\mathbb{N}$ is dense in $\beta \mathbb{N}$. Given a free ultrafilter $\omega$, for any $f\in l^{\infty}(\mathbb{N})$, there is a subsequence $\{m_{j}\}_{j=1}^{\infty}$ of $\mathbb{N}$ (depending on $f$) such that $\omega(f)=\lim_{j\rightarrow \infty}f(m_{j})$. We usually write $\omega(f)=\lim_{n\rightarrow \omega}f(n)$, called \emph{the limit of $f$ at $\omega$}.

For a C*-subalgebra $\mathcal{A}$ of $l^{\infty}(\mathbb{N})$, the linear functional $\rho$ is called a \emph{state} on $\mathcal{A}$ if $\rho(1)=1$ and $\rho(f)\geq 0$ for any $f\in \mathcal{A}$ with $f\geq 0$. We shall study the $A$-invariant states on anqies defined below.

\begin{definition}
Suppose that $\mathcal{A}$ is an $A$-invariant C*-subalgebra of $l^{\infty}(\mathbb{N})$, i.e., $Af\in \mathcal{A}$ for any $f\in \mathcal{A}$.
A state $\rho$ on $\mathcal{A}$ is called \textbf{$A$-invariant}, or ``invariant" for short, if $\rho(Af)=\rho(f)$ for any $f\in \mathcal{A}$.
\end{definition}
Invariant states may or may not be related to average values of functions. Here we give an example to explain this phenomena.
\begin{example}{\rm
Let $G=\cup_{n=1}^{\infty}\{n^2-n,n^2-n+1,\ldots,n^2-1\}$ be a subset of $\mathbb{N}$, and $G_{n}=\{i\in G: 0\leq i\leq n-1\}$. Define $F_{n}(f)=\frac{1}{|G_{n}|}\sum_{i\in G_{n}}f(i)$ for $f\in l^{\infty}(\mathbb{N})$.
Then for each given $f$ the function $n\mapsto F_{n}(f)$ gives rise to a function in $l^{\infty}(\mathbb{N})$.
Choose $\omega\in \beta\mathbb{N}\setminus \mathbb{N}$, and define $F_{\omega}(f)=\lim_{n\rightarrow \omega}F_{n}(f)$.
Then $F_{\omega}$ is an $A$-invariant state on $l^{\infty}(\mathbb{N})$. If $\chi_G$ is the
characteristic function supported on $G$, then $F_\omega(\chi_G)=1$.
But the relative density of $G$ in $\N$ is zero. Thus $F_\omega(f)$
does not depend on the average sum $\frac1n\sum_{i=0}^{n-1}f(i)$.
}\end{example}

On the other hand, there are $A$-invariant states depending on average values of functions, which are called ``mean states" in \cite[Definition 5.3]{LG}.
\begin{definition}\label{def of mean state}
Suppose $\omega\in\beta\N\setminus \N$ is a given
free ultrafilter. For any $n\in\N$ and any $f$ in $l^\infty(\N)$, we
define $E_n(f)=\frac 1{n}\sum_{j=0}^{n-1}f(j)$. Then, for each given $f$, the function $n\to E_n(f)$ gives rise to another function in
$l^\infty(\N)$. The limit of $E_n(f)$ at $\omega$ is denoted by
$E_\omega(f)$. Then $E_\omega$ is an $A$-invariant state defined on $l^\infty(\N)$
or called ``a\textbf{ mean state}''
$($or, ``a\textbf{ mean}" for short$)$.
\end{definition}
From now on, we shall use $E$ to denote a given mean state on $l^{\infty}(\mathbb{N})$ (depending on a free ultrafilter). For a real-valued function $f\in l^\infty(\mathbb{N})$, we always have:
$$
\liminf_{n\to \infty} \frac 1{n}\ds\sum_{j=0}^{n-1}f(j) \le E(f) \le
\limsup_{n\to \infty} \frac 1{n}\ds\sum_{j=0}^{n-1}f(j).
$$

Suppose $\mathcal{A}$ is a countably generated  anqie of $\mathbb{N}$.  For each $N\in \mathbb{N}$,
define the state $\rho_{N}$ on $\mathcal{A}$ by $\rho_{N}(f)=\frac{1}{N}\sum_{j=0}^{N-1}f(j)$.
So $\{\rho_{N}\}_{N=1}^{\infty}$ is a sequence in $(\mathcal{A}^{\sharp})_{1}$.
By Proposition \ref{metrizable},  $(\mathcal{A}^{\sharp})_{1}$ is metrizable and compact, so there is a subsequence $\{\rho_{N_{m}}\}_{m=1}^{\infty}$ that
converges to some $\rho\in (\mathcal{A}^{\sharp})_{1}$.  We call this $\rho$ {\sl the limit of $\rho_{N_m}$ or the state given (uniquely) by the sequence $\{N_{m}\}_{m=1}^{\infty}$}. It is not hard to check that $\rho$ is an $A$-invariant state, and for any free ultrafilter $\omega$ in the closure of $\{N_{m}: m=1,2,3,\ldots\}$ in $\beta\mathbb{N}$, the restriction of $E_{\omega}$ on $\mathcal{A}$ is $\rho$.

Now we perform the GNS construction on $l^\infty(\N)$ with respect
to $E$.
Define $\<f, g\>_{E}=E(\bar{g}f)$, the semi-inner product
on $l^\infty(\N)$ and $\|f\|_{E}=(\<f,f\>_{E})^{\frac{1}{2}}$, the semi-norm on $l^\infty(\N)$ (see \cite[Proposition 4.3.1]{KR}).
We use $\mathcal{K}$ to denote the subalgebra of
$l^\infty(\N)$ containing all $f$ so that $E(|f|^2)=\<f,f\>_{E}=0$. Then
$\mathcal{K}$ is a closed two-sided ideal in $l^\infty(\N)$. Thus
$\B:=l^\infty(\N)/\mathcal{K}$ is a C*-algebra, and $\<\ ,\
\>_{E}$ induces an inner product on $\B$.
For $f\in l^\infty(\N)$, we
may use $\tilde{f}$ (or simply $f$ if there is no ambiguity) to
denote the coset $f+\mathcal{K}$ in $\B$. When
$\tilde{f},\tilde{g}\in\B$, we still use
\begin{equation}\label{inner product}
\<\tilde{f},\tilde{g}\>_{E}=E(f\overline{g})=\lim_{n\to \omega} \frac 1{n}\sum_{j=0}^{n-1}f(j)\overline{g}(j)
\end{equation}
to denote the inner product on
$\B$ and
\begin{equation}\label{inner norm}
\|\tilde{f}\|_E=(\<\tilde{f}, \tilde{f}\>_{E})^{\frac12}=\Big(\lim_{n\to \omega} \frac 1{n}\sum_{j=0}^{n-1}|f(j)|^2\Big)^{\frac{1}{2}}
\end{equation}
for
the (Hilbert space) vector norm on $\B$. The completion of $\B$
under this norm is denoted by $\mathcal{H}_E$.
\begin{remark}{\rm Our later results will depend on $E$ but not on a specific one. Therefore,
our definitions or properties stated later are for any mean state $E$.
For example, if $f$ and $g$ are \emph{orthogonal}, $E(f\overline{g})=0$ holds for any mean state $E$.
The orthogonality of arithmetic functions may be viewed as disjointness between two functions in number theory.}
\end{remark}
Next, we study some properties of (strongly) asymptotically periodic functions (Definitions 1.3 and 1.4). We start from the following generalized notion of periodicity that introduced in \cite[Section 5]{LG}.
An arithmetic function $f\in l^{\infty}(\mathbb{N})$ is said to be\emph{ essentially periodic} (or ``e-periodic")
if there is an integer $n_{0}\geq 1$ such that $f=A^{n_{0}}f$ in $\mathcal{H}_{E}$.
The smallest such $n_{0}$ $(\geq 1)$ is called the e-period of $f$. From the definition of strongly asymptotically periodic functions, it is easy to see that e-periodic functions belong exactly to this class.

In the following, we use $e(x)$ to denote $e^{2\pi ix}$ for simplicity, and $1_{S}$ to denote the indicator of a predicate $S$,
that is $1_{S}=1$ when $S$ is true and $1_{S}=0$ when $S$ is false. It is not hard to check that $e(\sqrt{n})$ is an e-periodic function of e-period 1. Note that arithmetic functions satisfying $f(n)=f(n+1)$ for all $n$ must be constant ones.
Thus e-periodic functions are far from periodic ones. In the following, we shall construct e-periodic functions with e-period $k$,
for any $k\geq 1$.
\begin{example}\label{random asymptotically periodic functions}{\rm
Let $\{m_{j}\}_{j=1}^{\infty}$ and $\{n_{j}\}_{j=1}^{\infty}$ be two sequences of positive integers with
$\lim_{j\rightarrow \infty}m_{j}=\lim_{j\rightarrow \infty}n_{j}=\infty$.
For any given $q$, choose $\alpha=\{0,1,\ldots,1\},\beta=\{1,0,\ldots,0\}$ as two vectors of length $q$. We construct the function $f$ (written as $\{f(n)\}_{n=1}^{\infty}$) successively:
 \[\underbrace{\alpha\alpha\cdot\cdot\cdot\alpha}_{m_{1}}\underbrace{\beta\beta\cdot\cdot\cdot\beta}_{n_{1}}
\underbrace{\alpha\alpha\cdot\cdot\cdot\alpha}_{m_{2}}\underbrace{\beta\beta\cdot\cdot\cdot\beta}_{n_{2}}\cdot\cdot\cdot.\]
Then $f$ is an e-periodic function with e-period $q$, and so a strongly asymptotically periodic function.}
\end{example}
The proof of the above fact is more involved. Here are some details.

By $\lim_{j\rightarrow \infty}m_{j}=\lim_{j\rightarrow \infty}n_{j}=\infty$, for any $\epsilon>0$,
there is an $j_{0}$ such that $qm_{j}, qn_{j}>\frac{1}{\epsilon}+1$  when $j> j_{0}$.
Then the number of $n$ between $1$ and $N$ satisfying $f(n+q)\neq f(n)$ is less than $2qj_{0}+qN\epsilon$.
Moreover, $\lim_{N\rightarrow \infty}\frac{1}{N}\sum_{n=1}^{N}|f(n+q)-f(n)|^2\leq \lim_{N\rightarrow \infty}(\frac{2qj_{0}}{N}+q\varepsilon)\leq q\varepsilon$.
Since $\epsilon$ is arbitrarily small, $\lim_{N\rightarrow \infty}\frac{1}{N}\sum_{n=1}^{N}|f(n+q)-f(n)|^2=0$.
It is easy to see that for any positive integer $l\leq q-1$, $\lim_{N\rightarrow \infty}\frac{1}{N}\sum_{n=1}^{N}|f(n+l)-f(n)|^2\neq 0$.
Hence $f$ is an e-periodic function with e-period $q$.

In Example \ref{random asymptotically periodic functions},
if the two sequences $\{m_{j}\}_{j=1}^{\infty}$ and $\{n_{j}\}_{j=1}^{\infty}$ further satisfy
$\lim_{j\rightarrow \infty}\frac{m_{j}}{n_{j}}=a\neq 0$, then we can show that $f$ is not the weak limit of periodic functions. That is for any mean state $E$, there does not exist a sequence $\{f_{n}\}_{n=1}^{\infty}$ of periodic functions,  such that the limit of $f-f_{n}$ is zero in $\mathcal{H}_{E}$.

Using a similar argument to the proof of Example \ref{random asymptotically periodic functions}, we have the following result.
\begin{example}\label{reflect the distribution of moebius in short intervals}
Let $\{N_{j}\}_{j=0}^{\infty}$ be a sequence of natural numbers with $N_{0}=0$ and $\lim_{j\rightarrow \infty}(N_{j+1}-N_{j})=\infty$. Let $q\geq 1$ and $a\geq 0$ be given integers. Define $f(n)$ to be $a_{j}1_{n\equiv a(mod~q)}$ when $N_{j}\leq n<N_{j+1}$ for $j=0,1,\ldots$, where $\{a_{j}\}_{j=0}^{\infty}$ is a sequence of complex numbers with $\sup_{j}|a_{j}|<\infty$. Then $f$ is an e-periodic function with e-period $q$, and so a strongly asymptotically periodic function.
\end{example}
By Definitions \ref{def:asymptotically periodic} and \ref{def:strongly asymptotically periodic}, e-periodic and strongly asymptotically periodic functions are asymptotically periodic. There are many asymptotically periodic functions that are far from e-periodic ones. For example, if $\theta$ is irrational then $f(n)=e(n\theta)$ is asymptotically periodic. But it is not the weak (or $l^2$-) limit of e-periodic functions. In fact, we have the following result.
\begin{proposition}
Let $\theta$ be an irrational number and  $f(n)=e(n\theta)$. Then $f$ is orthogonal to all e-periodic functions, that is for any e-periodic function $g$, $E(f\overline{g})=0$ holds for all mean states $E$.
\end{proposition}
\begin{proof}
Suppose that the e-period of $g$ is $k$. Given a mean state $E$, by the $A$-invariance of $E$, $E(f\overline{g})=\<f,g\>_{E}=\<A^{lk}f, A^{lk}g\>_{E}=\<A^{lk}f, g\>_{E}$ for any $l\geq 1$. Thus $\<f,g\>_{E}=\<\frac{1}{m}\sum_{l=1}^{m}A^{lk}f,g\>_{E}$.
For any $\epsilon> 0$, we can choose a sufficiently large integer $m$ such that $|\frac{1}{m}\sum_{l=1}^{m}e((n+lk)\theta)|=|\frac{1}{m}\sum_{l=1}^{m}e(lk\theta)|<\epsilon$ for any $n\in \N$. Hence $\|\frac{1}{m}\sum_{l=1}^{m}A^{lk}f\|_{E}<\epsilon$.
It follows from the Cauchy-Schwarz inequality that $|\<f,g\>_{E}|< \epsilon\|g\|_{l^{\infty}}$. Letting $\epsilon\rightarrow 0$. Then $\<f,g\>_{E}=0$.
\end{proof}
Next, we provide another example of strongly asymptotically periodic function.
\begin{example}
Let $f\in l^{\infty}(\mathbb{N})$. Suppose that the closure of $\{(f(n),f(n+1),\ldots):n\in \mathbb{N}\}$ in $\prod_{\mathbb{N}}\overline{f(\mathbb{N})}$, endowed with the product topology, is countable. Then $f$ is a strongly asymptotically periodic function.
\end{example}
In the following, we give some detailed argument for the above nontrivial fact.
Denote $\mathcal{A}_{f}$ as the anqie generated by $f$ and
$X_{f}$ as the maximal ideal space of $\mathcal{A}_{f}$. By Theorem \ref{the method to compute the anqie}, $X_{f}$ is homeomorphic to the closure of $\{(f(n),f(n+1),\ldots):n\in \mathbb{N}\}$ in $\prod_{\mathbb{N}}\overline{f(\mathbb{N})}$ and so is a countable space. Let $E$ be a mean state on $l^{\infty}(\mathbb{N})$. Then the restriction $E$ on $\mathcal{A}_{f}$ is an invariant state on $\mathcal{A}_{f}$. By Theorem \ref{the measure induced by a state}, there is a $\sigma_{A}$-invariant probability measure $\nu$ on $X_{f}$ such that
$E(g)=\int_{X_{f}}g(x)~d\nu$  for any $g\in \mathcal{A}_{f}$. Since $X_{f}$ is a countable and compact metric space, $\nu$ must be an atomic measure. Assume that $\nu$ is supported at $x_1,x_2,\ldots$ in $X_f$.
For each $x_i$, there are two nature numbers $s_{i}$ and $t_{i}$ such that $A^{-s_{i}}\{x_i\}\cap A^{-t_{i}}\{x_i\}
\neq \emptyset$.
Thus there is a $k_{i}$ such that $A^{k_{i}}x_{i}=x_{i}$. Set $n_{j}=\prod_{i=1}^{j}k_{i}$. So for each $l\geq 1$,
\begin{eqnarray*}
       \|f-A^{ln_{j}}f\|^2_{E} &=& E(|f-A^{ln_{j}}f|^2) \\
       &=&\int_{X_{f}}|(f-f\circ A^{ln_{j}})(x)|^2~d\nu \\
        &=&\sum_{m=1}^{\infty}|f(x_{m})-f\circ A^{ln_{j}}(x_{m})|^2~\nu(\{x_m\})  \\
        &=&\sum_{m=j+1}^{\infty} |f(x_{m})-f\circ A^{ln_{j}}(x_{m})|^2~\nu(\{x_{m}\}).
     \end{eqnarray*}
Then
$$\|f-A^{ln_{j}}f\|^2_{E} \leq
4\|f\|_{l^{\infty}}^2
\sum_{m=j+1}^{\infty}\nu(\{x_{m}\})\rightarrow 0
$$
as $j$ goes to $\infty$. Hence $f$ is a strongly asymptotically periodic function.

\begin{proposition}\label{the tested property of moebius}
The M\"obius function is disjoint from all strongly asymptotically periodic functions if and only if  for any given integer $q\geq 1$,
\begin{equation}\label{20210409}
\limsup_{N\rightarrow \infty}\frac{1}{N}\sum_{n=1}^{N}\Bigg|\sum_{l=1}^{h}\mu(n+ql)\Bigg|^2=o(h^2).
\end{equation}
\end{proposition}
\begin{proof}
We first prove $``\Rightarrow "$ part. Given integers $q\geq 1$ and $a\geq 0$. Take $a_{j}=e(\theta_{j})$ such that \[\sum_{N_{j}\leq n< N_{j+1}}\mu(n)e(\theta_{j})1_{n\equiv a(mod~q)}=\Bigg|\sum_{\substack{N_{j}\leq n< N_{j+1}\\n\equiv a(mod~q)}}\mu(n)\Bigg|.\] Then the M\"obius disjointness of $f(n)$ for any $f(n)$ defined as in Example \ref{reflect the distribution of moebius in short intervals} is equivalent to \[\lim_{m\rightarrow \infty}\frac{1}{N_{m}}\sum_{j=0}^{m-1}\Bigg|\sum_{\substack{N_{j}\leq n<N_{j+1}\\n\equiv a(mod~q)}}\mu(n)\Bigg|=0\] for any sequence $\{N_{j}\}_{j=0}^{\infty}$ with $N_{0}=0$ and $\lim_{j\rightarrow \infty}(N_{j+1}-N_{j})=\infty$. This is further equivalent to (see e.g., \cite[Lemma 5.2]{GW})
\begin{equation}\label{20210410}
\lim_{h\rightarrow \infty}\limsup_{N\rightarrow \infty}\frac{1}{Nh}\sum_{m=1}^{N} \Bigg|\sum_{\substack{n=m\\n\equiv a(mod~q)}}^{m+h}\mu(n)\Bigg|=0.
\end{equation}
It is not hard to check that for $N$ large enough,
\begin{equation}\label{an identity0409}
\sum_{n=1}^{N}\Bigg|\sum_{l=1}^{h}\mu(n+ql)\Bigg|=\frac{1}{q}\sum_{a=1}^{q}
  \sum_{m=1}^{N}\Bigg|\sum_{\substack{n=m\\n\equiv a(mod~q)}}^{m+hq}\mu(n)\Bigg|+O(N).
\end{equation}
Note that $q$ is given and by the trivial estimate,
\[\limsup_{N\rightarrow \infty}\frac{1}{N}\sum_{n=1}^{N}\Bigg|\frac{1}{h}\sum_{l=1}^{h}\mu(n+ql)\Bigg|^2\leq \limsup_{N\rightarrow \infty}\frac{1}{N}\sum_{n=1}^{N}\Bigg|\frac{1}{h}\sum_{l=1}^{h}\mu(n+ql)\Bigg|.\]
By equations (\ref{20210410}) and (\ref{an identity0409}), we obtain equation (\ref{20210409}).

We then prove $``\Leftarrow "$ part. Let $f(n)$ be a strongly asymptotically periodic function. It suffices to show that for any mean state $E$, $\<f,\mu\>_{E}=E(f\mu)=0$. By the strongly asymptotical periodicity of $f$,
for any $\epsilon>0$, there is a positive integer $n_{0}$ such that
\begin{equation}\label{1}
\|f-A^{ln_{0}}f\|_{E}< \epsilon
\end{equation}
for any $l\in \N$. By equation (\ref{20210409}), there is a sufficiently large $l_{0}$ such that
\[\limsup_{N\rightarrow \infty}\frac{1}{N}\sum_{n=1}^{N}\Bigg|\frac{1}{l_{0}}\sum_{l=1}^{l_{0}}\mu(n+ln_{0})\Bigg|^2<\epsilon.\]
This implies
\begin{equation}\label{20409}
\|\frac{1}{l_{0}}\sum_{l=1}^{l_{0}}A^{ln_{0}}\mu\|_{E}< \epsilon.
\end{equation}
Note that for any $l\in \mathbb{N}$, $\<f,\mu\>_{E}=\<A^{ln_{0}}f,A^{ln_{0}}\mu\>_{E}=\<A^{ln_{0}}f-f,A^{ln_{0}}\mu\>_{E}+\<f,A^{ln_{0}}\mu\>_{E}$.
Then
\begin{equation}\label{30409}
\<f,\mu\>_{E}=\frac{1}{l_{0}}\sum_{l=1}^{l_{0}}\<A^{ln_{0}}f-f,A^{ln_{0}}\mu\>_{E}+\<f,\frac{1}{l_{0}}\sum_{l=1}^{l_{0}}A^{ln_{0}}\mu\>_{E}.
\end{equation}
By the Cauchy-Schwarz inequality and equations (\ref{1}), (\ref{20409}), (\ref{30409}), we conclude that $|\<f,\mu\>_{E}|< \epsilon(\|f\|_{l^{\infty}}+1)$ for any $\epsilon>0$. Letting $\epsilon\rightarrow 0$, $\<f,\mu\>_{E}=0$.
\end{proof}

There are many arithmetic functions that are not asymptotically periodic, such as $f(n)=e(n^2\theta)$ with $\theta$ irrational. This follows from the fact $\|f-A^{m}f\|_{E}=\sqrt{2}$ for any $m\geq 1$ and any mean state $E$.
Moreover, this function is orthogonal to all asymptotically periodic functions, which is claimed in \cite[Theorem 5.9]{LG} without proof. Here we give the proof in the following theorem.
\begin{theorem}\label{f is orthogonal to any asymptotically periodic function}
Let $f(n)=e(n^2\theta)$ with $\theta$ irrational. We have

(\romannumeral1) For any $l\neq m$, $\<A^{l}f,A^{m}f\>_{E}=0$ for any mean state $E$.

(\romannumeral2) The function $f$ is orthogonal to all asymptotically periodic functions in $l^{\infty}(\mathbb{N})$.
\end{theorem}
\begin{proof}
(\romannumeral1) This result is true as
\[
\<A^{l}f,A^{m}f\>_{E}=e((l^2-m^2)\theta)\cdot\lim_{N\rightarrow \infty}\frac{1}{N}\sum_{j=1}^{N}e((2l-2m)j\theta)=0.
\]
(\romannumeral2) Let $g$ be an asymptotically periodic function.
Assume on the contrary that $|\<f,g\>_{E}|>\delta$ for some  $\delta>0$. By definition, there is a sequence $\{n_{j}\}_{j=1}^{\infty}$
such that  $\lim_{j\rightarrow \infty}\|A^{n_{j}}g-g\|_{E}=0$.
So when $j$ is large enough, $\|A^{n_{j}}g-g\|_{E}<{\delta}/{2}$. By the Cauchy-Schwarz inequality and the fact that $\|f\|_E=1$, we have
\begin{eqnarray*}
      |\<g,A^{n_{j}}f\>_{E}| &=& |\<g-A^{n_{j}}g,A^{n_{j}}f\>_{E}+\<A^{n_{j}}g,A^{n_{j}}f\>_{E}| \\
       &\geq & |\<A^{n_{j}}g,A^{n_{j}}f\>_{E}|-|\<g-A^{n_{j}}g,A^{n_{j}}f\>_{E}| \\
       &>& {\delta}/{2}.
\end{eqnarray*}
It follows from (\romannumeral1) that the set $\{A^{n_{j}}f: j=1,2,\ldots\}$ is an orthogonal set in $\mathcal{H}_{E}$.
By Bessel's inequality, $\|g\|_{E}^2\geq \sum_{j=1}^{\infty}|\<g,A^{n_{j}}f\>_{E}|^2=\infty$. This contradicts the fact that $g\in l^\infty(\N)$. Hence $\<f,g\>_{E}=0$ for any mean state $E$.
\end{proof}

\begin{remark}{\rm
Based on the above proof,
an arithmetic function is orthogonal to all asymptotically periodic functions if it satisfies condition (\romannumeral1) of Theorem \ref{f is orthogonal to any asymptotically periodic function}.
}\end{remark}

\begin{theorem}\label{the product of the translations of the Moebius function}
Let $r\geq 2$ and $\mu_{r}(n)=1$ if $n$ is $r$-th power-free and zero otherwise. For any $s\geq 1$ and $m_{1},\ldots,m_{s}\in \mathbb{N}$, $\prod_{i=1}^{s}A^{m_{i}}(\mu_{r})$ is strongly asymptotically periodic.
\end{theorem}

\begin{proof}
It suffices to prove that $\mu_{r}$ is strongly asymptotically periodic.
Let $p_j$ be the $j$-th prime, define the sequence $\{n_j\}_{j=1}^\infty$ by  $n_{j}=p_{1}^rp_{2}^r\cdots p_{j}^r$.
By \cite{MIR}, for any positive integer $m$, we have for any mean state $E$,
\[
\<\mu_{r},A^{m}\mu_{r}\>_{E}=\lim_{N\rightarrow \infty}\frac{1}{N}\sum_{n=1}^{N}\mu_{r}(n)\mu_{r}(n+m)=\prod_{p}(1-\frac{2}{p^r})\prod_{p^r|m}(1+\frac{1}{p^r-2}).
\]
So for any positive integer $l$, $\<\mu_{r},A^{ln_{j}}\mu_{r}\>_{E}\geq \<\mu_{r},A^{n_{j}}\mu_{r}\>_{E}$. Moreover,
\begin{align*}
  \|\mu_{r}-A^{ln_{j}}\mu_{r}\|_{E}^2\leq & \|\mu_{r}-A^{n_{j}}\mu_{r}\|_{E}^2=2\<\mu_{r},\mu_{r}\>_{E}-2\<\mu_{r},A^{n_{j}}\mu_{r}\>_{E} \\
  =&2\sum_{n=1}^{\infty}\frac{\mu(n)}{n^{r}}-2\prod_{p}(1-\frac{1}{p^r})\prod_{p>p_{j}}(1+\frac{1}{p^r-2})^{-1}\\
  =&2\prod_{p}(1-\frac{1}{p^r})(1-\prod_{p>p_{j}}
(1+\frac{1}{p^r-2})^{-1})
\end{align*}
tends to $0$ when $j$ goes to infinity. Then $\{\mu_{r}-A^{ln_{j}}\mu_{r}\}_{j=1}^{\infty}$ converges to zero in $\mathcal{H}_{E}$ uniformly with respect to all $l\in\N$.\end{proof}

Now we prove Proposition \ref{non zero entropy example}.
\begin{proof}[Proof of Proposition \ref{non zero entropy example}]
By Theorem \ref{the product of the translations of the Moebius function}, $f=A^{m_{1}}(\mu^2)\cdots A^{m_{s}}(\mu^2)$ is strongly asymptotically periodic. Let $(X_{f},\sigma_{A})$ be the anqie generated by $f$. Then by Theorem \ref{the method to compute the anqie}, we describe $X_{f}$ as a closed subspace $\widetilde{X_{f}}$ of $\{0,1\}^{\mathbb{N}}$ and represent $\sigma_{f}$ as the Bernoulli shift $B$ on the space. It follows from \cite{Sar} that $(\widetilde{X_{f}},B)$ has positive topological entropy and then $(X_{f},\sigma_{A})$ has positive entropy. Assume on the contrary that there is a dynamical system $(X,T)$ with the topological entropy of $T$ zero, such that $f(n)=F(T^{n}x_{0})$ for some $F\in C(X)$ and $x_{0}\in X$. By Proposition \ref{a factor}, the topological entropy of $T$ is greater than or equal to that of $\sigma_{A}$. This contradicts the assumption that the topological entropy of $T$ is zero. Hence $f(n)$ cannot be realized in any dynamical systems with zero topological entropy.
\end{proof}

In the next section, we shall see that the anqie of $\mathbb{N}$ generated by any asymptotically periodic function
is closely related to a rigid dynamical system.
\section{$\sigma_{A}$-invariant measures}\label{mean states}
\label{A-invariant measure induced by a state}
Suppose that $(X,\sigma_{A})$ (or $\A$) is an anqie of $\mathbb{N}$. It is a basic fact that a continuous map on $X$ is always (Borel) measurable. Then $\sigma_{A}$ is a measurable transformation.
We call a (Borel) measure $\nu$ on $X$ \emph{$\sigma_{A}$-invariant} if for any Borel set $F$ of $X$, $\nu(F)=\nu((\sigma_{A})^{-1}F)$.
In the following, we show that for any given invariant state on $\A$, there is an induced $\sigma_{A}$-invariant Borel probability measure on $X$.
\begin{theorem}\label{the measure induced by a state}
Let $(X,\sigma_{A})$ (or $\A$) be an anqie of $\mathbb{N}$. Suppose $\rho$ is an invariant state on $\mathcal{A}$.
Then there is a unique $\sigma_{A}$-invariant Borel probability measure $\nu$ on $X$ such that for any $g\in \mathcal{A}$,
\begin{equation}\label{the representation}
\rho(g)=\int_{X}g(x)~d\nu,
\end{equation}
where $g(x)$ is the image of $g$ under the Gelfand transform $($see equation (\ref{Gelfand transform})$)$.
\end{theorem}

\begin{proof}
Since $\rho$ is an invariant state on $\mathcal{A}$ and $\A\cong C(X)$, $\rho$ can be viewed as a state on $C(X)$
satisfying $\rho(f\circ \sigma_{A})=\rho(f)$ for any $f\in C(X)$.
By the Riesz representation theorem, there is a unique Borel measure $\nu$ on $X$, such that for any $f\in C(X)$,
\begin{equation}
\label{Riesz representation theorem}
\rho(f)=\int_{X}f(x)~d\nu.
\end{equation}
Moreover, $\nu$ is regular and it has the property that for any compact subset $K\subset X$,
\begin{equation}\label{the measure of K}
\nu(K)=\inf\{\rho(h):h\in C(X),h|_{K}=1\}.
\end{equation}
In the following, we show that $\nu(\sigma_{A}^{-1}(F))=\nu(F)$ for any Borel set $F$.
First we prove that if $\nu(F)=0$, then $\nu(\sigma_{A}^{-1}(F))=0$. In fact, by the regularity of $\nu$, for any $\epsilon>0$,
there is a compact set $K\subset \sigma_{A}^{-1}(F)$ such that
\begin{equation}\label{regularity inequality}
\nu(\sigma_{A}^{-1}(F))<\nu(K)+\epsilon.
\end{equation}
Note that $\sigma_{A}(K)\subset F$. So $\nu(\sigma_{A}(K))=0$. By equation (\ref{the measure of K}),
there is an $h\in C(X)$ such that $h|_{\sigma_{A}(K)}=1$ and $\rho(h)<\epsilon$. By equation (\ref{the measure of K}) again,
$\nu(K)\leq \rho(h\circ \sigma_{A})=\rho(h)< \epsilon$. Then $\nu(\sigma_{A}^{-1}(F))<2\epsilon$ by equation (\ref{regularity inequality}). Since $\epsilon$ can be arbitrarily small, $\nu(\sigma_{A}^{-1}(F))=0$.

Now we assume that $\nu(F)\neq 0$.   By Lusin's Theorem, there is a sequence $\{f_{n}\}_{n=1}^{\infty}$ in $C(X)$
and a Borel set $G$ with $\nu(G)=0$, such that $\|f_{n}\|\leq 1$ and  $\lim_{n\rightarrow \infty}f_{n}(x)=\chi_{F}(x)$ for any $x\in X\setminus G$,
where $\chi_{F}$ is the characteristic function supported on $F$. Thus $\lim_{n\rightarrow \infty}f_{n}\circ \sigma_{A}(x)=\chi_{F}\circ \sigma_{A}(x)$ for $x\in X\setminus \sigma_{A}^{-1}G$.
By the analysis in the above paragraph, $\nu(\sigma_{A}^{-1}G)=0$. By equation (\ref{Riesz representation theorem}) and the Lebesgue Dominated Convergence Theorem,
\[\lim_{n\rightarrow \infty}\rho(f_{n})=\lim_{n\rightarrow \infty}\int_{X}f_{n}(x)~d\nu=\nu(F)\]
and
\[\lim_{n\rightarrow \infty}
\rho(f_{n}\circ \sigma_{A})=\lim_{n\rightarrow \infty}\int_{X}f_{n}\circ \sigma_{A}(x)~d\nu=\nu(\sigma_{A}^{-1}(F)).\]
Since $\rho(f_{n})=\rho(f_{n}\circ A)$, we obtain $\nu(F)=\nu(\sigma_{A}^{-1}(F))$ as claimed.

Finally, it follows from $\rho(1)=1$ that $\nu(X)=1$. Thus $\nu$ is a $\sigma_{A}$-invariant Borel probability measure on $X$.
\end{proof}
We call the $\sigma_{A}$-invariant (Borel) probability measure $\nu$ given by equation (\ref{the representation}) \emph{the measure induced by $\rho$}. Suppose $\mathcal{A}$ is a countably generated anqie of $\mathbb{N}$, then by Proposition \ref{metrizable}, $(\mathcal{A}^{\sharp})_{1}$ is a compact metrizable space.
Hence for any mean state $E$ on $l^{\infty}(\N)$, there is a sequence $\{N_{m}\}_{m=1}^{\infty}$ of positive integers such that for any $g\in \A$,
$E(g)=\lim_{m\rightarrow \infty}\frac{1}{N_{m}}\sum_{n=0}^{N_{m}-1}g(n)$.
By Theorem \ref{the measure induced by a state}, there is a $\sigma_{A}$-invariant probability measure $\nu$ on $X$ such that for any $g\in \A$,
\begin{equation}\label{generic measure}
\lim_{m\rightarrow \infty}\frac{1}{N_{m}}\sum_{n=0}^{N_{m}-1}g(n)=\int_{X}g(x)d\nu.
\end{equation}
On the other hand, for any $x\in X$ define a Borel probability measure $\delta_{x}$ on $X$ such that for any Borel set $B$ in $\B$, $\delta_{x}(B)=1$ if $x\in B$, and $0$ otherwise. For each $N\geq 1$, define
\begin{equation}\label{eqDirac}
\delta_{N,x}=\frac{1}{N}\sum_{n=0}^{N-1}\delta_{(\sigma_{A})^{n}x}.
\end{equation}
It is easy to check that $\delta_{N,x}$ is a Borel probability measure on $X$.
Now fix $x=\iota(0)$, i.e., the multiplicative state on $\mathcal{A}$ given by $\iota(0): f\mapsto f(0)$ for any $f\in \mathcal{A}$. Then
\[\int_{X}g(x)d\delta_{N_{m},\iota(0)}=\frac{1}{N_{m}}\sum_{n=0}^{N_{m}-1}\int_{X}g(x)d\delta_{(\sigma_{A})^{n}(\iota(0))}=\frac{1}{N_{m}}\sum_{n=0}^{N_{m}-1}g(n)\]
holds for any $g\in \mathcal{A}$, correspondingly $g(x)\in C(X)$.
By equation (\ref{generic measure}),
\[\lim_{m\rightarrow \infty}\int_{X}g(x)d\delta_{N_{m},\iota(0)}=\int_{X}g(x)d\nu.\]
We call $\nu$ the \emph{(weak*) limit} of $\delta_{N_{m},\iota(0)}$. In \cite{EJMT5}, $\iota(0)$ is also called the \emph{quasi-generic for $\nu$ along $\{N_{m}\}_{m=1}^{\infty}$}.
\begin{remark} {\rm
Let $\mathcal{A}_{f}$ be the anqie generated by $f$ and $X_{f}$ the maximal ideal space of $\mathcal{A}_{f}$.
By Theorem \ref{the method to compute the anqie},
$X_{f}$ is the closure of $\{(f(n),f(n+1),\ldots):n\in \mathbb{N}\}$ in $\prod_{\mathbb{N}}\overline{f(\mathbb{N})}$.
Suppose that $\rho$ is an invariant state on $\mathcal{A}_{f}$, and $\nu$ the measure induced by $\rho$.
Naturally,  $\nu$ can be extended to a probability measure (denote by $\widetilde{\nu}$) on $\prod_{\mathbb{N}}\overline{f(\mathbb{N})}$, which is defined by $\widetilde{\nu}(F)=\nu(F\cap X_{f})$ for any Borel set $F$ of $\prod_{\mathbb{N}}\overline{f(\mathbb{N})}$.
It is easy to see that $\widetilde{\nu}$ is $B$-invariant, where $B$ is the Bernoulli shift on $\prod_{\mathbb{N}}\overline{f(\mathbb{N})}$.
In the following, we still use $\nu$ to denote $\widetilde{\nu}$ if it makes no ambiguity.
}\end{remark}
Next, we discuss the connection between asymptotically periodic functions and rigid dynamical systems.
\begin{theorem}\label{A is rigid}
Suppose that $f$ is an asymptotically periodic function. Let $(X_{f},\sigma_{A})$ (or, $\mathcal{A}_{f}$) be the anqie generated by $f$. Let $E$ be a mean state and $\nu$ the measure induced by $E$ on $X_{f}$. Then $(X_{f},\nu,\sigma_{A})$ is rigid.
\end{theorem}

\begin{proof}
By the definition of asymptotically periodic function, there is a sequence of positive integers $\{n_{j}\}_{j=1}^{\infty}$ with
$\lim_{j\rightarrow \infty}E(|A^{n_{j}}f-f|^2)=0$.
It is not hard to check that for any $g\in \mathcal{A}_{f}$, $\lim_{j\rightarrow \infty}E(|A^{n_{j}}g-g|^2)=0$.
Thus for any $g(x)\in C(X_{f})$, by equation (\ref{the representation}), we have
\begin{equation}\label{the integral expression}
\lim_{j\rightarrow \infty}\int_{X_{f}}|g\circ (\sigma_{A})^{n_{j}}(x)-g(x)|^2d\nu=0.
\end{equation}
By Lusin's Theorem, for any Borel set $F$ of $X_{f}$, there is a sequence $\{g_{n}\}_{n=1}^{\infty}$ of continuous functions on $X_{f}$ with $\|g_{n}\|\leq 1$ such that  $\lim_{n\rightarrow \infty}g_{n}(x)=\chi_{F}(x)$ for almost all $x$. Then by Lebesgue's Dominated Convergence Theorem, for any $\epsilon>0$,
there is a $g_{n_{0}}$ such that $\int_{X_{f}}|\chi_{F}(x)-g_{n_{0}}(x)|d\nu<{\epsilon}/{3}$. By equation (\ref{the integral expression}),
there is a sufficiently large $K$ such that $\int_{X_{f}}|g_{n_{0}}\circ (\sigma_{A})^{n_{j}}(x)-g_{n_{0}}(x)|^2d\nu<{\epsilon}/{3}$ when $j>K$.
Thus
\begin{eqnarray*}
  \nu((\sigma_{A})^{-n_{j}}F\triangle F)&=&\int_{X_{f}}|\chi_{F}\circ (\sigma_{A})^{n_{j}}(x)-\chi_{F}(x)|d\nu \\
&\leq&\int_{X_{f}}|\chi_{F}\circ (\sigma_{A})^{n_{j}}(x)-g_{n_{0}}\circ (\sigma_{A})^{n_{j}}(x)|d\nu \\
  && + \int_{X_{f}}|g_{n_{0}}\circ (\sigma_{A})^{n_{j}}(x)-g_{n_{0}}(x)|d\nu
+\int_{X_{f}}|g_{n_{0}}(x)-\chi_{F}(x)|d\nu<\epsilon.
\end{eqnarray*}
Hence $\lim_{j\rightarrow \infty}\nu((\sigma_{A})^{-n_{j}}F\triangle F)=0$.
\end{proof}
It is known that for a measure-preserving dynamical system $(X,\nu,T)$ with $T$ a homeomorphism, if $\nu$ has discrete spectrum, then $T$ is rigid (see e.g., \cite{KP7})). In the following, we show that this rigidity satisfies conditions (\ref{restriction1}), (\ref{asymptotical periodicity1}) in Theorem \ref{comparison result}.
\begin{proposition}\label{discre spectrum}
Let $X$ be a compact metric space and $T:X\rightarrow X$ be a homeomorphism. Let $\nu$ be a $T$-invariant probability measure on $X$. Suppose that $\nu$ has discrete spectrum. Then $(X,\nu,T)$ satisfies conditions (\ref{restriction1}), (\ref{asymptotical periodicity1}) in Theorem \ref{comparison result}. That is, for any $g(x) \in L^2(X,\nu)$, there are sequences $\{h_j\}_{j=1}^\infty$ and $\{n_j\}_{j=1}^\infty$ of positive integers with
\begin{equation}\label{5/19formula1}
\lim_{j\rightarrow \infty} \frac{\log\log h_j}{\log h_j} \frac{n_j}{\varphi(n_j)} = 0
\end{equation}
such that
\begin{equation}\label{5/19formula2}
\lim_{j\rightarrow \infty} \frac{1}{h_j} \sum_{l=1}^{h_j} \|g \circ T^{ln_j} - g\|_{L^2(\nu)}^2 = 0.
\end{equation}
\end{proposition}

\begin{proof}
Since $\nu$ has discrete spectrum, by definition there is a standard orthogonal basis $\{g_s(x)\}_{s=1}^\infty$ in $L^2(X,\nu)$ with $g_s(Tx) = e^{2\pi i \lambda_s } g_s(x)$ for some real number $\lambda_s$, where $s=1,2,\ldots.$ Let $g(x)\neq 0 \in L^2(X,\nu)$, write $g(x) = \sum_{s=1}^\infty a_s g_s(x)$. Then
$\|g\|_{L^2(\nu)}^2 = \sum_{s=1}^\infty |a_s|^2 < \infty.$
For $j=1,2,\ldots,$ choose $\epsilon_j = \|g\|^2_{L^2(\nu)}/j$ and $N_j\geq 1$ with
$
\sum_{s=N_j+1}^\infty |a_s|^2 < \frac{\epsilon_j}{8}.
$
Let $t_j = 2\|g\|^2_{L^2(\nu)} e^{N_j}/\epsilon_j$. Choose $n_j$ such that
$
|e^{2\pi i n_j \lambda_s} - 1| \leq 1/t_{j}
$
for $s=1,\ldots,N_j$, where $1\leq n_j \leq t_j^{N_j}$.  Let $h_{j}=t_{j}^{1/2}$.
By the choice of $n_j$ and $h_{j}$, as well as the estimate $\frac{n_{j}}{\varphi(n_{j})}\ll \log\log n_{j}$, it is not hard to check that they satisfy condition (\ref{5/19formula1}).
%\beas
%$\frac{n_j}{\varphi(n_j)} \frac{\log\log h_j}{\log h_j}
%\ll (\log\log n_j) \frac{\log\log h_j}{\log h_j}
%\ll (\log N_j+\log\log t_j) \frac{\log\log h_j}{\log h_j}  \rightarrow 0$
%%\eeas
%as $j \rightarrow \infty$.
Then
\beas
&& \frac{1}{h_j} \sum_{l=1}^{h_j} \|g\circ T^{ln_j} - g\|^2_{L^2(\nu)}=\|\sum_{s=1}^\infty a_s g_s \circ T^{ln_j}(x) - \sum_{s=1}^\infty a_s g_s(x)\|^2_{L^2(\nu)} \\
&=& \frac{1}{h_j} \sum_{l=1}^{h_j} \sum_{s=1}^\infty |a_s|^2 | e^{2\pi il n_j \lambda_s } -1 |^2 \\
&=& \frac{1}{h_j} \sum_{l=1}^{h_j} \sum_{s=1}^{N_j} |a_s|^2 | e^{2\pi il n_j \lambda_s } -1 |^2 + \frac{1}{h_j} \sum_{l=1}^{h_j} \sum_{s=N_j+1}^{\infty} |a_s|^2 | e^{2\pi il n_j \lambda_s } -1 |^2 \\
&\leq&  \frac{1}{h_j} \sum_{l=1}^{h_j} N_j \|g\|^2_{L^2(\nu)} \frac{l^2}{t_j^2} + \frac{\epsilon_j}{2} \leq \frac{N_j\|g\|^2_{L^2(\nu)}}{t_j} + \frac{\epsilon_j}{2} < \epsilon_j \rightarrow 0, ~~~~~~as~j \rightarrow \infty.
\eeas
\end{proof}
\begin{proposition}\label{WAP is asymptotically periodic function}
Let $f\in l^{\infty}(\mathbb{Z})$ be a weakly almost periodic function (see Remark \ref{WAP} for definition). Then $f_{1}$ $($=$f$ restricted to $\mathbb{N}$$)$ belongs to the class of asymptotically periodic functions described by conditions (\ref{restriction3/1}) and (\ref{asymptotical periodicity3/1}).
That is, for any mean state $E$, there are sequences $\{h_{j}\}_{j=1}^{\infty}$ and $\{n_{j}\}_{j=1}^{\infty}$ of positive integers with
\begin{equation}\label{5/19formula5}
\lim_{j\rightarrow \infty}\frac{\log\log h_{j}}{\log h_{j}}\frac{n_{j}}{\varphi(n_{j})}=0
\end{equation}
such that
\begin{equation}\label{5/19formula6}
\lim_{j\rightarrow \infty}\frac{1}{h_{j}}\sum_{l=1}^{h_{j}}E(|f-A^{ln_{j}}f|^2)=0.
\end{equation}
\end{proposition}
\begin{proof}
Let $\mathcal{A}_{f}$ be the C*-algebra of $l^{\infty}(\mathbb{Z})$ generated by $1$ and $\{A^{n}f:n\geq 0\}$. We use $X_{f}$ to denote the maximal ideal space of $\mathcal{A}_{f}$ and $\sigma_{A}$ the homeomorphism on $X_{f}$ induced by $n\mapsto n+1$ on $\mathbb{Z}$. Let $E$ be a mean state on $l^{\infty}(\N)$ depending on $\omega$ (see Definition \ref{def of mean state}), where $\omega$ is in the weak* closure of the sequence $\{N_{m}\}_{m=1}^{\infty}$ of positive integers in $\mathbb{\beta}\mathbb{N}$. Then $E$ can be naturally treated as a state on $\mathcal{A}_{f}$ in the way that
\[g\mapsto E(g_{1})\] for any $g(n)\in \mathcal{A}_{f}$, where $g_{1}(n)$ is the restriction of $g(n)$ to $\mathbb{N}$. Define the state $\rho_{N_{m}}$ on $\mathcal{A}_{f}$ by $\rho_{N_{m}}(g)=\frac{1}{N_{m}}\sum_{n=0}^{N_{m}-1}g(n)$ for any $g\in\mathcal{A}_{f}$. So $E$ is in the weak* closure of $\{\rho_{N_{m}}\}_{m=1}^{\infty}$ in $(\mathcal{A}_{f}^{\sharp})_{1}$. Since $(\mathcal{A}_{f}^{\sharp})_{1}$ is compact and metrizable, there is a subsequence $\{N_{m_{s}}\}_{s=1}^{\infty}$ of positive integers satisfying for any $g\in \A_{f}$,
\begin{equation}\label{the state defined on the subsequence}
E(g)=\lim_{s\rightarrow \infty}\frac{1}{N_{m_{s}}}\sum_{n=0}^{N_{m_{s}}-1}g(n).
\end{equation}
By Theorem \ref{the measure induced by a state}, there is a $\sigma_{A}$-invariant probability measure $\nu$ on $X_{f}$ such that for any $g\in \A_{f}$,
\begin{equation}\label{generic measure2}
\lim_{s\rightarrow \infty}\frac{1}{N_{m_{s}}}\sum_{n=0}^{N_{m_{s}}-1}g(n)=\int_{X_{f}}\widetilde{g}(x)d\nu,
\end{equation}
where $\widetilde{g}(x)$ is the image of $g(n)$ under the Gelfand transform.
Thanks to \cite[Proposition 5.1]{HWY}, $\nu$ has discrete spectrum. By Proposition \ref{discre spectrum}, there are sequences $\{n_{j}\}_{j=1}^{\infty}$ and $\{h_{j}\}_{j=1}^{\infty}$ of positive integers satisfying condition (\ref{5/19formula5}) such that
\[\lim_{j\rightarrow \infty}\frac{1}{h_{j}}\sum_{l=1}^{h_{j}}\|\widetilde{f}((\sigma_{A})^{ln_{j}}x)-\widetilde{f}(x)\|_{L^{2}(\nu)}^2=0.\]
By equations (\ref{the state defined on the subsequence}) and (\ref{generic measure2}),
\[\lim_{j\rightarrow \infty}\frac{1}{h_{j}}\sum_{l=1}^{h_{j}}E(|f-A^{ln_{j}}f|^2)=\lim_{j\rightarrow \infty}\frac{1}{h_{j}}\sum_{l=1}^{h_{j}}\int_{X_{f}}|\widetilde{f}((\sigma_{A})^{ln_{j}}x)-\widetilde{f}(x)|^2d\nu=0.\]
\end{proof}
To summarize, through invariant states we establish a connection between arithmetics and measurable dynamics.
Specifically, for any given arithmetic function $f$ in $l^{\infty}(\mathbb{N})$, it corresponds to a measure-preserving dynamical system $(X_{f},\nu, \sigma_{A})$.
So we can apply tools in ergodic theory to study the system $(X_{f},\nu, \sigma_{A})$ and further study properties of $f$.
\section{Proofs of Theorem \ref{main estimate}}\label{estimate of second moment}
Theorem \ref{main estimate} comes from a general result on the average of bounded multiplicative functions in short arithmetic progressions (see Proposition \ref{main theorem2} below). In the statement of the next result, we shall use the following distance function of Granville and Soundararajan,
\[\mathbb{D}_{k}(f(n),g(n);x):=\Big(\sum_{\substack{p\leq x\\p \nmid k}}\frac{1-\text{Re}(f(p)\overline{g(p)})}{p}\Big)^{\frac{1}{2}}\]
for two multiplicative functions $f(n)$ and $g(n)$ with $|f(n)|,|g(n)|\leq 1$ for all $n\geq 1$. This distance function was used in \cite{BGS} to measure the pretentiousness between any multiplicative function $f(n)$ and some function for which exceptional modulus $k$ does exist. Throughout define
\begin{align*}
 M_{k}(f;x;T):&=\inf_{|t|\leq T}\mathbb{D}_{k}(f,n\mapsto n^{it};x)^2,\\
 M_{k}(f;k;x;T):&=\inf_{\chi(mod~k)}\inf_{|t|\leq T}\mathbb{D}_{k}(f\chi,n\mapsto n^{it};x)^2.
\end{align*}
\begin{proposition}\label{main theorem2}
Let $X$ be large enough with $1\leq k\leq (\log X)^{1/32}$. Let $3\leq h\leq X/k$. Let $f(n)$ be a multiplicative function with $|f(n)|\leq 1$ for all $n\geq 1$. Then
\begin{equation}\label{average on multiplicative function}
\sum_{\substack{a=1\\(a,k)=1}}^{k}\sum_{x=X}^{2X}\Bigg|\sum_{\substack{n=x\\n\equiv a(mod~k)}}^{x+hk}f(n)\Bigg|^2 \ll  h^2X\varphi(k)\Big(\frac{k}{\varphi(k)}\frac{\log \log h}{\log h}+\frac{1}{(\log X)^{1/300}}+\frac{M_{k}(f;k;X;2X)+1}{\exp(M_{k}(f;k;X;2X))}\Big).
\end{equation}
\end{proposition}
The major ingredient of our proof of the above result is Matom\"{a}ki-Radziwi{\l}{\l}'s estimate \cite{KM} on averages of multiplicative functions in short intervals and the large sieve. We defer the proof to Appendix C. Here we list some recent results on averages of multiplicative functions in short arithmetic progressions: the method used in \cite{KMT17} can give that the coefficient before $\frac{\log \log h}{\log h}$ is $k$ in formula (\ref{average on multiplicative function}); \cite[Theorem 3.1]{Kan} gave the result that when $f=\mu(n)$, then for any $\epsilon>0$, the left hand side of formula (\ref{average on multiplicative function})$\leq \epsilon h^2X\varphi(k)$, whenever $\sum_{p|k}1/p\leq (1-\epsilon)\sum_{p\leq h}1/p$;
For general multiplicative function, \cite[Theorem 1.6, Corollary 1.7]{KMT} gave the result that for any $\epsilon>0$, the left hand side of formula (\ref{average on multiplicative function})$\leq \epsilon h^2X\varphi(k)$ when $k$ is $h^{\epsilon^2}$-typical (i.e., there are not many prime factors of $k$ less than $h^{\epsilon^2}$).

The reason that we give the estimate in form of formula (\ref{average on multiplicative function}) is that our main interest is to concern about when $k$ is far larger than $h$,  whether the first term of the right hand side of formula (\ref{average on multiplicative function}) is still $h^2X\varphi(k)o_{h}(1)$. According to our result, this is true if $k$ is as large as $\exp(h^{o(1)})$ since $k/\varphi(k)\ll \log\log k$. We believe that the coefficient before $(\log \log h)/\log h(=o_{h}(1))$ in formula (\ref{average on multiplicative function}) can be as small as $O(1)$, an absolute constant independent of $k$. Actually, note that
\[\frac{1}{X} \sum_{n=1}^{X} \left|\sum_{l=1}^h \mu(n+lk)\right|^2 = \frac{1}{Xk} \sum_{a=1}^{k} \sum_{x=1}^{X} \Bigg|\sum_{\substack{n=x\\ n \equiv a(mod~k)} }^{x+hk} \mu(n)\Bigg|^2 + O(1).\]
It is likely to believe that
\begin{equation}\label{formula0410}
\limsup_{X \rightarrow \infty} \frac{1}{Xk} \sum_{a=1}^{k} \sum_{x=1}^{X} \Bigg|\sum_{\substack{n=x\\n\equiv a(mod~k)} }^{x+hk} \mu(n)\Bigg|^2= o(h^2),
\end{equation}
where the little ``o" term is independent of $k\geq 1$. This is implied by a positive answer to the Chowla conjecture. For a general non-pretentious multiplicative function $f(n)$ with $|f(n)|\leq 1$ for any $n\in \mathbb{N}$, equation (\ref{formula0410}) in which $\mu(n)$ is replaced by $f(n)$  would be implied by a positive answer to Elliott's conjecture (see \cite[Conjecture II]{Elliot}, \cite{KMT17}).

As an application of Proposition \ref{main theorem2}, we shall prove certain self correlations of the M\"{o}bius function which is stated in Theorem \ref{main estimate}. In this proof, we need the following known result about the non-pretentious nature of $\mu(n)1_{(n,k)=1}$ (see e.g., \cite[Lemma C.1]{KMT17}).
\begin{lemma}\label{the pretentious distance of the moebius function}
Let $X$ be large enough with $k\leq \log X$. Let $f(n)=\mu(n)1_{(n,k)=1}$. Then
\[
\inf_{1\leq d\leq k}M_{k}(f;d;X;2X) \geq (1/3-\epsilon) \log\log X+O(1),
\]
where $\epsilon>0$ is sufficiently small.
\end{lemma}
\begin{proof}[Proof of Theorem \ref{main estimate}]
Given $k\geq 1$ and $h\geq 2$. For $X$ large enough with $\log X>h^2k$,
\begin{align}
  &\sum_{n=X}^{2X}|\sum_{l=1}^{h}\mu(n+kl)|^2=\sum_{a=1}^{k}\sum_{\substack{n=X\\n\equiv a(mod~k)}}^{2X}|\sum_{l=1}^{h}\mu(n+kl)|^2 \nonumber \\
  &=\sum_{a=1}^{k}
  \sum_{m=X/k}^{2X/k}|\sum_{l=1}^{h}\mu(km+kl+a)|^2+O(h^2k)
  =\sum_{a=1}^{k}
  \sum_{m=X/k}^{2X/k}|\sum_{\substack{n=(m+1)k\\n\equiv a(mod~k)}}^{(m+h+1)k}\mu(n)|^2+O(h^2k) \nonumber \\
&=\sum_{a=1}^{k}
  \sum_{\substack{x=X\\k|x}}^{2X}|\sum_{\substack{n=x\\n\equiv a(mod~k)}}^{x+hk}\mu(n)|^2+O(h^2k)
=\frac{1}{k}\sum_{a=1}^{k}
  \sum_{x=X}^{2X}|\sum_{\substack{n=x\\n\equiv a(mod~k)}}^{x+hk}\mu(n)|^2+O(X)\\
&=\frac{1}{k}\sum_{d|k}\sum_{\substack{a=1\\(a,k/d)=1}}^{k/d}\sum_{x=X}^{2X}|\sum_{\substack{n=x/d\\n\equiv a(mod~k/d)}}^{x/d+hk/d}\mu(dn)|^2+O(X)\nonumber \\
&=\frac{1}{k}\sum_{d|k}d\sum_{\substack{a=1\\(a,k/d)=1}}^{k/d}\sum_{x=X/d}^{2X/d}\Big|\sum_{\substack{n=x\\n\equiv a(mod~k/d)}}^{x+hk/d}\mu(n)1_{(n,k)=1}(n)\Big|^2+O(X\frac{1}{k}\sum_{d|k}\varphi(k/d))+O(X)\nonumber \\
&=\frac{1}{k}\sum_{d|k}d\sum_{\substack{a=1\\(a,k/d)=1}}^{k/d}\sum_{x=X/d}^{2X/d}\Big|\sum_{\substack{n=x\\n\equiv a(mod~k/d)}}^{x+hk/d}\mu(n)1_{(n,k)=1}(n)\Big|^2+O(X). \nonumber
\end{align}
Summarize the above, we have
\begin{equation}\label{eqi}
 \sum_{n=X}^{2X}|\sum_{l=1}^{h}\mu(n+kl)|^2=\frac{1}{k}\sum_{d|k}d\sum_{\substack{a=1\\(a,k/d)=1}}^{k/d}\sum_{x=X/d}^{2X/d}\Big|\sum_{\substack{n=x\\n\equiv a(mod~k/d)}}^{x+hk/d}\mu(n)1_{(n,k)=1}(n)\Big|^2+O(X).
\end{equation}
Then for $X$ large enough, by Proposition \ref{main theorem2} and Lemma \ref{the pretentious distance of the moebius function},
\begin{align*}
  &\sum_{n=X}^{2X}|\sum_{l=1}^{h}\mu(n+kl)|^2\ll \frac{1}{k}(\sum_{d|k}d) h^2X/d\varphi(k/d)\Big(\frac{k/d}{\varphi(k/d)}\frac{\log\log h}{\log h}+\frac{1}{(\log X)^{1/400}}\Big)\\
  &=h^2X(\sum_{d|k}\frac{1}{d})\frac{\log\log h}{\log h}+\frac{h^2X}{(\log X)^{1/400}}\leq h^2X \prod_{p|k}(1-1/p)^{-1} \frac{\log\log h}{\log h}+\frac{h^2X}{(\log X)^{1/400}}\\
  &\leq h^2X \frac{k}{\varphi(k)}\frac{\log\log h}{\log h}+\frac{h^2X}{(\log X)^{1/400}}.
\end{align*}
Hence
\begin{equation}\label{main estimate 2}
\limsup_{N\rightarrow \infty}\frac{1}{N}\sum_{n=1}^{N}|\sum_{l=1}^{h}\mu(n+kl)|^2\ll h^2\frac{k}{\varphi(k)}\frac{\log\log h}{\log h},
\end{equation}
as claimed.
\end{proof}
\section{Proofs of Theorems \ref{disjointfromasymptoticallyperiodicfunctions3/1}, \ref{main theorem 2}, and Proposition \ref{investigation of Problem 1}}\label{proof of two theorems}
As an application of Theorem \ref{main estimate}, at the beginning of this section, we prove
that the M\"{o}bius function is disjoint from certain asymptotically periodic functions (i.e., Theorem \ref{disjointfromasymptoticallyperiodicfunctions3/1}).
\begin{proof}[Proof of Theorem \ref{disjointfromasymptoticallyperiodicfunctions3/1}]
Assume on the contrary, there is an $f\in l^{\infty}(\mathbb{N})$ with conditions (\ref{restriction3/1}) and (\ref{asymptotical periodicity3/1}) such that
$\lim_{N\rightarrow \infty}\frac{1}{N}\sum_{n=1}^{N}\mu(n)f(n)\neq 0$,
then there is a constant $c_{0}>0$ and a mean state $E$ such that
\begin{equation}\label{3/11formula1}
|\<f,\mu\>_{E}|\geq c_{0}
\end{equation}
By conditions (\ref{restriction3/1}) and (\ref{asymptotical periodicity3/1}), there are correspondingly sequences $\{h_{j}\}_{j=0}^{\infty}$ and $\{n_{j}\}_{j=0}^{\infty}$ of positive integers with
\begin{equation}\label{restriction3/1formula2}
\lim_{j\rightarrow \infty}\frac{\log\log h_{j}}{\log h_{j}}\frac{n_{j}}{\varphi(n_{j})}=0
\end{equation}
and
\begin{equation}\label{asymptotical periodicity3/1formula3}
\lim_{j\rightarrow \infty}\frac{1}{h_{j}}\sum_{l=1}^{h_{j}}E(|f-A^{ln_{j}}f|^2)=0.
\end{equation}
Let $\delta=\frac{c_{0}}{2(\|f\|_{l^{\infty}}+1)}$. By Theorem \ref{main estimate}, formulas (\ref{restriction3/1formula2}) and (\ref{asymptotical periodicity3/1formula3}), there is a $k_{0}$ such that
\[\frac{1}{h_{k_{0}}}\sum_{l=1}^{h_{k_{0}}}\|A^{ln_{k_{0}}}f-f\|_{E}^2<\delta^2\] and
\[\|\frac{1}{h_{k_{0}}}\sum_{l=1}^{h_{k_{0}}}A^{ln_{k_{0}}}\mu\|_{E}^2<\delta^2.\]
For any $l\in \mathbb{N}$, $\<f,\mu\>_{E}=\<A^{ln_{k_{0}}}f,A^{ln_{k_{0}}}\mu\>_{E}=\<A^{ln_{k_{0}}}f-f,A^{ln_{k_{0}}}\mu\>_{E}
+\<f,A^{ln_{k_{0}}}\mu\>_{E}$. Then
\begin{align*}
 |\<f,\mu\>_{E}|=&\Big|\frac{1}{h_{k_{0}}}\sum_{l=1}^{h_{k_{0}}}\<A^{ln_{k_{0}}}f-f,A^{ln_{k_{0}}}\mu\>_{E}+\<f,\frac{1}{h_{k_{0}}}
 \sum_{l=1}^{h_{k_{0}}}A^{ln_{k_{0}}}\mu\>_{E}\Big| \\
  \leq &\frac{1}{h_{k_{0}}}\sum_{l=1}^{h_{k_{0}}}\|A^{ln_{k_{0}}}f-f\|_{E}\cdot\|\mu\|_{E}+
  \|\frac{1}{h_{k_{0}}}\sum_{l=1}^{h_{k_{0}}}A^{ln_{k_{0}}}\mu\|_{E}\cdot\|f\|_{E}\\
 \leq& \Big(\frac{1}{h_{k_{0}}}\sum_{l=1}^{h_{k_{0}}}\|A^{ln_{k_{0}}}f-f\|_{E}^2\Big)^{\frac{1}{2}}+
 \|\frac{1}{h_{k_{0}}}\sum_{l=1}^{h_{k_{0}}}A^{ln_{k_{0}}}\mu\|_{E}\cdot\|f\|_{E}\\
 \leq &\delta(\|f\|_{l^{\infty}}+1)=c_{0}/2.
\end{align*}
Here we applied the Cauchy-Schwarz inequality to the first and second inequalities in the above, and the fact that $\|\mu\|_{E}\leq \|\mu\|_{l^{\infty}}=1$ and $\|f\|_{E}\leq \|f\|_{l^{\infty}}$. This contradicts formula (\ref{3/11formula1}). Hence the claim in this theorem holds.
\end{proof}

Now we prove Proposition \ref{investigation of Problem 1}.
\begin{proof}[Proof of Proposition \ref{investigation of Problem 1}]
Assume on the contrary that Problem 1 does not hold, that is, there is an asymptotically periodic function $f(n)$ such that
$\lim_{N\rightarrow \infty}\frac{1}{N}\sum_{n=1}^{N}\mu(n)f(n)\neq 0$. Then there is a $c_{0}>0$, a mean state $E$ and a sequence $\{n_{j}\}_{j=1}^{\infty}$ of positive numbers such that
\begin{equation}\label{5five}
|\<\mu,f\>_{E}|\geq c_{0}
\end{equation}
and
\begin{equation}\label{7seven}
\lim_{j\rightarrow\infty}\|A^{n_{j}}f-f\|_{E}^2=0.
\end{equation}
Let $\delta=\frac{c_{0}}{2(\|f\|_{l^{\infty}}+1)}$. By formula (\ref{condition independent of k}), choose a sufficiently large $l_{0}$ with
\begin{equation}
\|\frac{1}{l_{0}}\sum_{l=1}^{l_{0}}A^{lk}\mu\|_{E}<\delta,
\end{equation}
for any $k\geq 1$.
By equation (\ref{7seven}), there is an $n_{0}$ such that
\[\|A^{n_{0}}f-f\|_{E}<\frac{2\delta}{l_{0}+1}.\]
Then by the triangle inequality,
\begin{equation}
\frac{1}{l_{0}}\sum_{l=1}^{l_{0}}\|f-A^{ln_{0}}f\|_{E}\leq \frac{1}{l_{0}}\sum_{l=1}^{l_{0}}\sum_{j=1}^{l}\|A^{(j-1)n_{0}}f-A^{jn_{0}}f\|_{E}=\frac{1}{l_{0}}\sum_{l=1}^{l_{0}}\sum_{j=1}^{l}
\|f-A^{n_{0}}f\|_{E}<\delta.
\end{equation}
By the $A$-invariance of $E$ and the Cauchy-Schwarz inequality,
\begin{align*}
  \<f,\mu\>_{E}=&\frac{1}{l_{0}}\sum_{l=1}^{l_{0}}\<A^{ln_{0}}f-f,A^{ln_{0}}\mu\>_{E}
  +\<f,\frac{1}{l_{0}}\sum_{l=1}^{l_{0}}A^{ln_{0}}\mu\>_{E}\\
  & \leq \frac{1}{l_{0}}\sum_{l=1}^{l_{0}}\|f-A^{ln_{0}}f\|_{E}\cdot\|\mu\|_{l^{\infty}}+
  \|\frac{1}{l_{0}}\sum_{l=1}^{l_{0}}A^{ln_{0}}\mu\|_{E}\cdot\|f\|_{l^{\infty}}\\
  &<\delta(1+\|f\|_{l^{\infty}})=c_{0}/2.
\end{align*}
This contradicts formula (\ref{5five}). Hence formula (\ref{condition independent of k}) implies Problem 1.
\end{proof}
In the rest of this section, we shall prove Theorem \ref{main theorem 2}, which states that if SMDC holds, then $\mu$  is disjoint from all asymptotically periodic functions. Before proving it, we need some preparations. We first provide a property of asymptotically periodic functions.
\begin{proposition}\label{asymptotically periodic has zeo metric entropy}
Let $f$ be an asymptotically periodic function and $\rho$ an invariant state on $\mathcal{A}_{f}$. Then for the measure-preserving dynamical
system $(X_{f},\nu, \sigma_{A})$ with $\nu$ the probability measure induced by $\rho$ on $X_{f}$, the measure-theoretic entropy of $\sigma_{A}$ is zero.
\end{proposition}
The above proposition follows immediately from Theorem \ref{A is rigid} and \cite[Example 5.3.3]{KP7}.
The basic connection between topological entropy (denoted by $h(T)$) and measure-theoretic entropy (denoted by $h_{\nu}(T)$) is the variational principle (see, e.g., \cite[Theorem 8.6]{pw}).
It states that for any topological dynamical system $(X,T)$,
$h(T)=\sup\{h_\nu(T):\nu$ is a $T$-invariant Borel probability measure on $X\}$. By this principle, it is  easy to see that if $h(T)=0$, then $h_\nu(T)=0$ for any $T$-invariant probability measure $\nu$.

Here is an interesting example about topological entropy and measure-theoretic entropy. By Theorem \ref{the product of the translations of the Moebius function}, $\mu^2$ is an asymptotically periodic function. So by Proposition \ref{asymptotically periodic has zeo metric entropy}, for any measure induced by a mean state $\rho$ on $X_{\mu^2}$, the measure-theoretic entropy of $\sigma_{A}$ is zero. While Peckner proved in \cite{RP} that there is a $\sigma_{A}$-invariant measure on $X_{\mu^2}$ such that the measure-theoretic entropy of $\sigma_{A}$ is equal to $\frac{6}{\pi^2}\log 2$, which equals the topological entropy of $\sigma_{A}$. So the measure-theoretic entropy varies with respect to different measures.

The following lemma is a consequence of Proposition \ref{asymptotically periodic has zeo metric entropy} and \cite[Lemmas 4.28, 4.29]{EJMT5}, which are used to prove the equivalence between SMDC and the M\"{o}bius disjointness of completely deterministic sequences.
\begin{lemma}\label{zero entropy function approach}
Let $f$ be an asymptotically periodic function and $\mathcal{A}_{f}$ be the anqie generated by $f$. Suppose $\{N_{m}\}_{m=1}^{\infty}$ is a strictly increasing sequence of positive integers
such that $N_{m}|N_{m+1}$. Further suppose the sequence $\{N_{m}\}_{m=1}^{\infty}$ satisfies the condition that
there is an $A$-invariant state $\rho$ on $\A_{f}$, such that for any $h\in \A_{f}$, $\rho(h)=\lim_{m\rightarrow \infty}\frac{1}{N_{m}}\sum_{n=1}^{N_{m}}h(n)$.
Then for any $\epsilon>0$, there is an arithmetic function $g$ with finite range, and a subsequence $\{N_{m(l)}\}_{l=1}^{\infty}$ such that

(\romannumeral1) for $(X_{g},\sigma_{A})$ the anqie generated by $g$, the topological entropy of $\sigma_{A}$ is zero.

(\romannumeral2) $\frac{1}{N_{m(l)}}\sum_{n=1}^{N_{m(l)}}|f(n)-g(n)|<\epsilon.$
\end{lemma}
Based on such connections between asymptotically periodic functions and arithmetic functions with associated anqies having zero entropy, we are ready to prove Theorem \ref{main theorem 2}.
\begin{proof}[Proof of Theorem \ref{main theorem 2}]
Assume on the contrary that there is some asymptotically periodic function $f$ such that
$\lim_{N\rightarrow \infty}\frac{1}{N}\sum_{n=1}^{N}\mu(n)f(n)\neq 0$,
then there is a constant $c_{0}>0$ and an increasing sequence $\{N_{m}\}_{m=1}^{\infty}$ of positive integers with $N_{m}|N_{m+1}$ such that
\begin{equation}\label{one}
\frac{1}{N_{m}}\Bigg|\sum_{n=1}^{N_{m}}\mu(n)f(n)\Bigg|\geq c_{0}.
\end{equation}
For each $N_{m}$, define a state $\rho_{N_{m}}$ on $\mathcal{A}_{f}$ by $\rho_{N_{m}}(h)=\frac{1}{N_{m}}\sum_{n=1}^{N_{m}}h(n)$
for any $h\in \mathcal{A}_{f}$. It follows from Proposition \ref{metrizable} that there is a subsequence $\{\rho_{N_{m(l)}}\}_{l=1}^{\infty}$ and
a state $\rho$ on $\mathcal{A}_{f}$, such that $\rho(h)=\lim_{l\rightarrow \infty}\frac{1}{N_{m(l)}}\sum_{n=1}^{N_{m(l)}}h(n)$  for any $h\in \mathcal{A}_{f}$. Then $\rho$ is $A$-invariant.
By Lemma \ref{zero entropy function approach}, there is a $g(n)$ with the topological entropy of $(X_{g},\sigma_{A})$ zero,
and a subsequence of $\{N_{m(l)}\}_{l=1}^{\infty}$ (denoted by $\{N_{m(l)}\}_{l=1}^{\infty}$ again), such that
\begin{equation}\label{above two}
\frac{1}{N_{m(l)}}\sum_{n=1}^{N_{m(l)}}|f(n)-g(n)|<\frac{c_{0}}{2}.
\end{equation}
Applying Sarnak's M\"{o}bius Disjointness Conjecture to $(X_{g},\sigma_{A})$,
\begin{equation}\label{above three}
\lim_{l\rightarrow \infty}\frac{1}{N_{m(l)}}\sum_{n=1}^{N_{m(l)}}\mu(n)\widetilde{g}(A^{n}(\iota(0)))=\lim_{l\rightarrow \infty}\frac{1}{N_{m(l)}}\sum_{n=1}^{N_{m(l)}}\mu(n)g(n)=0,
\end{equation}
where $\widetilde{g}(x)$ is the image of $g(n)$ in $C(X_{g})$ under the Gelfand transform (see equation (\ref{Gelfand transform})).
By equations (\ref{above two}) and (\ref{above three}), we obtain a result which contradicts formula $(\ref{one})$.
Then $\mu$ is disjoint from all asymptotically periodic functions.
\end{proof}
\section{Disjointness of M\"{o}bius from rigid dynamical systems}\label{applications}
In this section, we shall prove Theorem \ref{comparison result}, Corollary \ref{comparison example} and Proposition \ref{investigation of Problem 2}.
\begin{proof}[Proof of Theorem \ref{comparison result}]
Assume on the contrary, there is an $f\in C(X)$ such that
\[\lim_{N\rightarrow \infty}\frac{1}{N}\sum_{n=1}^{N}\mu(n)f(T^{n}x_{0})\neq 0,\]
then there is a constant $c_{0}>0$ and an increasing sequence $\{N_{m}\}_{m=1}^{\infty}$ of positive integers such that
\begin{equation}\label{two}
\frac{1}{N_{m}}\Bigg|\sum_{n=1}^{N_{m}}\mu(n)f(T^{n}x_{0})\Bigg|\geq 2c_{0}.
\end{equation}
Since $X$ is a compact metric space, $C(X)$ is countably generated as an abelian C*-algebra. By Proposition \ref{metrizable}, there is a subsequence of $\{N_{m}\}_{m=1}^{\infty}$ (denoted by $\{N_{m}\}_{m=1}^{\infty}$ again for convenience) and a $T$-invariant measure $\nu$ on $X$, such that
$\nu_{N_{m}}=\frac{1}{N_{m}}\sum_{n=0}^{N_{m}-1}\delta_{T^{n}x_{0}}$ weak* converges to $\nu$ as $m\rightarrow \infty$, i.e., for any $f\in C(X)$,
\[\lim_{m\rightarrow \infty}\int_{X}f(x)d\nu_{N_{m}}=\lim_{m\rightarrow \infty}\frac{1}{N_{m}}\sum_{n=0}^{N_{m}-1}f(T^{n}x_{0})=\int_{X}f(x)d\nu.\]
By formula (\ref{two}) and the condition stated in this theorem, there is a $g\in C(X)$ and sequences $\{h_{j}\}_{j=1}^{\infty}$ and $\{n_{j}\}_{j=1}^{\infty}$ of positive integers with $\lim_{j\rightarrow \infty}\frac{\log\log h_{j}}{\log h_{j}}\frac{n_{j}}{\varphi(n_{j})}=0$, such that
\begin{equation}\label{three}
\lim_{j\rightarrow\infty}\frac{1}{h_{j}}\sum_{l=0}^{h_{j}-1}\|g\circ T^{ln_{j}}-g\|_{L^2(\nu)}^2=0,
\end{equation}
and
\begin{equation}\label{four}
\frac{1}{N_{m}}\Bigg|\sum_{n=1}^{N_{m}}\mu(n)g(T^{n}x_{0})\Bigg|\geq c_{0}.
\end{equation}
Choose a free ultrafilter $\omega$ in the closure of $\{N_{m}: m=1,2,3,\ldots\}$ in $\beta\mathbb{N}$. Then the mean state $E$ on $l^{\infty}(\mathbb{N})$ defined by $E(h)=\lim_{N_{m}\rightarrow \omega}\frac{1}{N_{m}}\sum_{n=0}^{N_{m}-1}h(n)$ for any $h\in l^\infty(\N)$ is $A$-invariant. Recall the GNS construction in Section 4, we use $\<\ ,\
\>_{E}$ and $\|\cdot\|_{E}$ to denote the inner product and norm induced by $E$ on $\mathcal{H}_{E}$, respectively (see equations (\ref{inner product}) and (\ref{inner norm})). Let $\widetilde{g}(n)=g(T^{n}x_{0})$. Then by equation (\ref{four}), we have
\begin{equation}\label{five}
|\<\widetilde{g},\mu\>_{E}|\geq c_{0}.
\end{equation}
For any $l=1,2,\ldots$,
note that
\[\|g\circ T^{ln_{j}}-g\|_{L^2(\nu)}^2=\lim_{m\rightarrow \infty}\frac{1}{N_{m}}\sum_{n=0}^{N_{m}-1}|g(T^{ln_{j}+n}x_{0})-g(T^{n}x_{0})|^2.\]
So by equation (\ref{three}),
\begin{equation}\label{seven}
\lim_{j\rightarrow\infty}\frac{1}{h_{j}}\sum_{l=0}^{h_{j}-1}\|A^{ln_{j}}\widetilde{g}-\widetilde{g}\|_{E}^2=0.
\end{equation}
By an argument similar to the proof in Theorem \ref{disjointfromasymptoticallyperiodicfunctions3/1}, we have $|\<\widetilde{g},\mu\>_{E}|\leq c_{0}/2$.
This contradicts formula (\ref{five}). Hence we obtain
\begin{equation} \label{0805formula8-new}
\lim_{N\rightarrow \infty}\frac{1}{N}\sum_{n=1}^{N}\mu(n)f(T^{n}x_{0})=0.
\end{equation}
This completes the proof of the first part of this theorem.

In the rest, we show the second part of the claim in this theorem, which states the above disjointness holds over short intervals in average, that is
\[
\lim_{h\rightarrow \infty}\limsup_{N\rightarrow \infty}\frac{1}{Nh}\sum_{n=1}^{N}\Big|\sum_{l=1}^{h}\mu(n+l)f(T^{n+l}x_{0})\Big|=0.
\]
It is not hard to check that the above is equivalent to
for any increasing sequence $ \{N_{j}\}_{j=0}^{\infty}$ of natural numbers with $N_0=0$ and $\lim_{j\rightarrow\infty} (N_{j+1}-N_j)=\infty$,
\[
\lim_{m\rightarrow \infty}\frac{1}{N_{m}}\sum_{j=0}^{m-1}\Bigg|\sum_{{N_{j}}\leq n<N_{j+1}}\mu(n)f(T^{n}x_{0})\Bigg|=0,
\]
(see e.g., \cite[Lemma 5.2]{GW}).
 Take $\{\theta_{j}\}_{j=0}^{\infty}$ such that \[\sum_{N_{j}\leq n< N_{j+1}}\mu(n)f(T^{n}x_{0})e(\theta_{j})=\Bigg|\sum_{N_{j}\leq n< N_{j+1}}\mu(n)f(T^{n}x_{0})\Bigg|.\]
Define $s(n)=e(\theta_{j})$ when $N_{j}\leq n<N_{j+1}$, $j=0,1,\ldots$. According to the above analysis, it suffices to prove that
\begin{equation}\label{0805formula4}
\lim_{N\rightarrow \infty}\frac{1}{N}\sum_{n=1}^{N}\mu(n)f(T^{n}x_{0})s(n)=0.
\end{equation}
Then $s(n)$ is an e-periodic function with e-period $1$. Namely, for any mean state $E$ and $l\in \mathbb{N}$,
\[E(|s(n+l)-s(n)|^2)=0.\]
Let $(X_{s},\sigma_{A})$ be the anqie generated by $s(n)$ and $\widetilde{s}(x)$ be the image of $s(n)$ in $C(X_{s})$ under the Gelfand transform. Let $\mathcal{G}$ be the algebra generated by $\{1,\widetilde{s}\circ (\sigma_{A})^{n}(x):n=0,1,\ldots\}$. Then $\mathcal{G}$ is dense in $C(X_{s})$. By Theorem \ref{A is rigid}, for any mean sate $E$, it induces a measure $\kappa$ in the weak* closure of $\{\frac{1}{N}\sum_{n=0}^{N-1}\delta_{(\sigma_{A})^{n}\iota(0)}: N=1,2,\ldots\}$ in the space of Borel probability measures on $X_{s}$ satisfying
\[
E(|s(n+l)-s(n)|^2)=\int_{X_{s}}|\widetilde{s}\circ (\sigma_{A})^{l}(x)-\widetilde{s}(x)|^2d\kappa=0
\]
for any $l\in \mathbb{N}$.
By the above equation and the triangle inequality, it is not hard to check that conditions (\ref{restriction1}) and (\ref{asymptotical periodicity1}) in Theorem \ref{comparison result} hold for $(X\times X_{s},T\times \sigma_{A}, (x_{0},\iota(0))$ with $\mathcal{F}\times \mathcal{G}$ a dense set in $C(X\times X_{s})$.
By a similar argument to prove (\ref{0805formula8-new}), we have
\[\lim_{N\rightarrow \infty}\frac{1}{N}\sum_{n=1}^{N}\mu(n)f(T^{n}x_{0})\widetilde{s}((\sigma_{A})^{n}\iota(0))=0.\]
Note that $\widetilde{s}((\sigma_{A})^{n}\iota(0))=s(n)$. We obtain equation (\ref{0805formula4}). Now we complete the proof of this theorem.
\end{proof}
\begin{remark}\label{extending conditions}
{\rm Both BPV rigidity and PR rigidity in Theorem \ref{polynomial rigidity} are included in conditions (\ref{restriction1}), (\ref{asymptotical periodicity1}) in Theorem \ref{comparison result}. Firstly, since $\frac{n_{j}}{\varphi(n_{j})}=\prod_{p|n_{j}}\frac{p}{p-1}=\prod_{p|n_{j}}(1-\frac{1}{p})^{-1}\ll \exp(\sum_{p|n_{j}}\frac{1}{p})=O(1)$ by the BPV  rigidity, $(\ref{restriction1})$ holds for any sequence $\{h_{j}\}_{j=1}^{\infty}$ with $\lim_{j\rightarrow \infty}h_{j}=\infty$. By BPV rigidity, there is a subsequence of $\{n_{j}\}_{j=1}^{\infty}$ (denoted by $\{n_{j}\}_{j=1}^{\infty}$ again for convenience) such that $\|g\circ T^{n_{j}}-g\|_{L^2(\nu)}\leq \frac{1}{2^{j}}$. Choose $h_{j}=j$. Then by the triangle inequality and $T$-invariance of $\nu$, $\|g\circ T^{ln_{j}}-g\|_{L^2(\nu)}\leq l\|g\circ T^{n_{j}}-g\|_{L^2(\nu)}$. So
\[\frac{1}{h_{j}}\sum_{l=1}^{h_{j}}\|g\circ T^{ln_{j}}-g\|_{L^2(\nu)}^2\leq \frac{j^2}{4^{j}}\rightarrow 0,~as~j\rightarrow \infty,\]
as claimed in formula  (\ref{asymptotical periodicity1}). Secondly, we explain that PR rigidity is a special case of (\ref{restriction1}) and (\ref{asymptotical periodicity1}). Let $h_{j}=n_{j}^{\delta}$. Then $\lim_{j\rightarrow \infty}\frac{\log\log h_{j}}{\log h_{j}}\frac{n_{j}}{\varphi(n_{j})}=0$ since $\frac{n_{j}}{\varphi(n_{j})}\ll \log\log n_{j}$. }
\end{remark}

Next, we give an example that satisfies conditions (\ref{restriction1}), (\ref{asymptotical periodicity1}), but not BPV rigidity and PR rigidity. Let $\eta=(\mu^2(0),\mu^2(1),\ldots)$ and $B$ the Bernoulli shift on $\{0,1\}^{\mathbb{N}}$. Let $X_{\eta}$ be the closure of $\{B^{n}\eta:n=0,1,\ldots\}$ in $\{0,1\}^{\mathbb{N}}$. We call $(X_{\eta}, B)$ the \emph{square-free flow}. The study of dynamical properties of the square-free flow have received much attention (see, e.g., \cite{CS,RP,Sar}). In \cite{Sar}, Sarnak proved that $(X_{\eta},B)$ is proximal (i.e., for any $x,y\in X_{\eta}$, $\inf_{n\geq 1}d(T^{n}x,T^{n}y)=0$) and it is topologically ergodic having topological entropy $\frac{6}{\pi^2}\log 2$.
As a result of Theorem \ref{comparison result}, we obtain the following M\"{o}bius disjointness for the square-free flow \footnote{There are some other methods to prove Corollary \ref{comparison example}. Our primary interest here is to provide an example that distinguish Theorem \ref{comparison result} we obtained from \cite[Theorem 2.1]{Kan} (presented in Theorem \ref{polynomial rigidity} in this paper).}.
\begin{corollary}\label{comparison example}
Let $(X_{\eta},B)$ be the square-free flow. Then for any $f\in C(X_{\eta})$,
\[\lim_{N\rightarrow \infty}\frac{1}{N}\sum_{n=1}^{N}\mu(n)f(B^{n}\eta)=0.\]
\end{corollary}
\begin{proof}
For $i=0,1,\ldots$, let $\pi_{i}: X_{\eta}\to \{0,1\}$ be the
projection map from $X_{\eta}$ onto its $i$-th coordinate. Let $\mathcal{F}$ be the *-subalgebra of $C(X_{\eta})$ generated by $\{\pi_{0},\pi_{1},\ldots\}$. By the Stone-Weierstrass theorem (see, e.g., \cite[Theorem 3.4.14]{KR}), $\mathcal{F}$ is dense in $C(X_{\eta})$. By \cite{Sar}, there is a $B$-invariant measure $\nu$ such that
$\frac{1}{N}\sum_{n=0}^{N-1}\delta_{B^{n}\eta}$ weak* converges to $\nu$ as $N\rightarrow \infty$. Let $p_l$ be the $l$-th prime and $n_{j}=p_{1}^2p_{2}^2\cdots p_{j}^2$. By an argument similar to the proof in Theorem \ref{the product of the translations of the Moebius function}, for $i=0,1,\ldots$,
\begin{align*}
 \|\pi_{i}\circ B^{ln_{j}}-\pi_{i}\|_{L^2(\nu)}^2&=\lim_{N\rightarrow \infty}\frac{1}{N}\sum_{n=0}^{N-1}|\pi_{i}(B^{ln_{j}+n}\eta)-\pi_{i}(B^{n}\eta)|^2\\
  &=\lim_{N\rightarrow \infty}\frac{1}{N}\sum_{n=0}^{N-1}|\mu^2(i+ln_{j}+n)-\mu^2(i+n)|^2\\
  &=\lim_{N\rightarrow \infty}\frac{1}{N}\sum_{n=0}^{N-1}|\mu^2(ln_{j}+n)-\mu^2(n)|^2\\
  &\leq \frac{12}{\pi^2}(1-\prod_{p>p_{j}}
(1+\frac{1}{p^2-2})^{-1}).
\end{align*}
Then, for any increasing sequence $\{h_{j}\}_{j=1}^{\infty}$ of positive integers,
\[\lim_{j\rightarrow\infty}\frac{1}{h_{j}}\sum_{l=0}^{h_{j}-1}\|\pi_{i}\circ B^{ln_{j}}-\pi_{i}\|_{L^2(\nu)}^2\leq \lim_{j\rightarrow\infty}\frac{1}{h_{j}}\sum_{l=0}^{h_{j}-1}\frac{12}{\pi^2}(1-\prod_{p>p_{j}}
(1+\frac{1}{p^2-2})^{-1})=0.\]
It is not hard to check that for any $g\in \mathcal{F}$,
\[\lim_{j\rightarrow\infty}\frac{1}{h_{j}}\sum_{l=0}^{h_{j}-1}\|g\circ B^{ln_{j}}-g\|_{L^2(\nu)}^2=0.\]
Hence by Theorem \ref{comparison result}, we obtain the claim in this corollary.
\end{proof}
\begin{remark}\label{comparison rate}{\rm
In the following, we explain that for any $\pi_{i}$, $i=0,1,\ldots$, in the above dense set $\mathcal{F}$ of $C(X)$, there is no sequence $\{n_{j}\}_{j=1}^{\infty}$ satisfying BPV and PR rigidity in Theorem \ref{polynomial rigidity}.

On one hand, by the argument in Corollary \ref{comparison example},
\[\|\pi_{i}\circ B^{n_{j}}-\pi_{i}\|_{L^2(\nu)}^2 =\frac{12}{\pi^2}\Big(1-\prod_{p^2\nmid n_{j}}
(1+\frac{1}{p^2-2})^{-1}\Big).\]
If $\lim_{j\rightarrow \infty}\|\pi_{i}\circ B^{n_{j}}-\pi_{i}\|_{L^2(\nu)}^2=0$, it is not hard to check that there is a subsequence $\{n_{j_{s}}\}_{s=1}^{\infty}$ with $p_{1}^2\cdot\cdot\cdot p_{s}^{2}|n_{j_{s}}$, where $p_{s}$ is the $s$-th prime.
Then $\sum_{p|n_{j_{s}}}\frac{1}{p}\geq \sum_{l\leq s}\frac{1}{p_{l}}\rightarrow \infty$ as $s\rightarrow \infty$ by Mertens' Theorem (see e.g., \cite{IK}). So $\{n_{j}\}_{j=1}^{\infty}$ does not satisfy BPR rigidity in Theorem \ref{polynomial rigidity}.

On the other hand, for a given $\delta>0$ and $(n_{j})^{\frac{\delta}{2}}\leq l\leq h_{j}=n_{j}^{\delta}$ with $j$ sufficiently large, note that the number of distinct prime factors of $ln_{j}$ is $O_{\delta}(\log n_{j})$, we have
\begin{align*}
 \|\pi_{i}\circ B^{ln_{j}}-\pi_{i}\|_{L^2(\nu)}^2 =&\frac{12}{\pi^2}\Big(1-\prod_{p^2\nmid ln_{j}}
(1+\frac{1}{p^2-2})^{-1}\Big)\\
\geq &\frac{12}{\pi^2}\Big(1-\prod_{p}
(1+\frac{1}{p^2-2})^{-1}\prod_{p^2\leq h_{j}}(1+\frac{1}{p^2-2})\prod_{\substack{p^2|ln_{j}\\p^2>h_{j}}}(1+\frac{1}{p^2-2})\Big)\\
= &\frac{12}{\pi^2}\Big(1-\prod_{p^2>h_{j}}
(1+\frac{1}{p^2-2})^{-1}\big(1+O_{\delta}(\frac{\log n_{j}}{h_{j}})\big)\Big)\gg \frac{1}{\sqrt{h_{j}}\log h_{j}}.
\end{align*}
Hence, $\lim_{j\rightarrow \infty}\sum_{l=1}^{h_{j}}\|\pi_{i}\circ B^{ln_{j}}-\pi_{i}\|_{L^2(\nu)}^2\neq 0$. So the sequence $\{n_{j}\}_{j=1}^{\infty}$ does not satisfy PR rigidity in Theorem \ref{polynomial rigidity}. }
\end{remark}
\begin{remark}\label{extending conditions519}{\rm
For the square-free flow $(X_{\eta},B)$, let $\nu$ be the $B$-invariant measure such that $\eta$ is generic for $\nu$. Then $\nu$ has discrete spectrum by \cite{CS}. From Corollary \ref{comparison example} and Remark \ref{comparison rate}, we know that $(X_{\eta},B,\nu)$ satisfies conditions (\ref{restriction1}), (\ref{asymptotical periodicity1}) in Theorem \ref{comparison result}, but not BPV rigidity and PR rigidity in Theorem \ref{polynomial rigidity}.
}
\end{remark}

By a similar argument to the proof of Corollary \ref{comparison example}, the conclusion also holds for $\eta$ replaced by the point $(\prod_{i=1}^{w}\mu_{r}(m_{i}),\prod_{i=1}^{w}\mu_{r}(m_{i}+1),\cdot\cdot\cdot,\prod_{i=1}^{w}\mu_{r}(m_{i}+n),\cdot\cdot\cdot)$, where
$r\geq 2$, $w\geq 1$ and $m_{1},\ldots,m_{w}\in \mathbb{N}$ are given, $\mu_{r}(n)=1$ if $n$ is $r$-th power-free and zero otherwise.

At the end, let us prove Proposition \ref{investigation of Problem 2}.
\begin{proof}[Proof of Proposition \ref{investigation of Problem 2}]
We first show that Problem 1 implies Problem 2. Let $f\in C(X)$. Then for any $\nu$ in the weak* closure of $\{\frac{1}{N}\sum_{n=0}^{N-1}\delta_{T^{n}x_{0}}: N=0,1,2,\ldots\}$ in the space of Borel probability measures on $X$, there is a sequence $\{n_{j}\}_{j=1}^{\infty}$ (may depend on $\nu$) of positive integers satisfying
\begin{equation}\label{formula1 in proposition 1.16}
\lim_{j\rightarrow\infty}\|f\circ T^{n_{j}}-f\|_{L^2(\nu)}^2=0.
\end{equation}
Let $g(n)=f(T^{n}x_{0})$. In the following, we want to show that $g(n)$ is an asymptotically periodic function. Let $\mathcal{A}_{g}$ be the anqie generated by $g(n)$ and $E$ be a mean state. Then there is a sequence $\{N_{m}\}_{m=1}^{\infty}$ of positive integers such that for any $h\in \A_{g}$,
$E(h)=\lim_{m\rightarrow \infty}\frac{1}{N_{m}}\sum_{n=0}^{N_{m}-1}h(n)$. By Theorem 5.1, there is a probability measure $\nu_{1}$ on $X_{f}$, such that
\begin{equation}\label{formula2 in proposition 1.16}
E(h)=\lim_{m\rightarrow \infty}\frac{1}{N_{m}}\sum_{n=0}^{N_{m}-1}h(n)=\int_{X_{f}}h(x)d\nu_{1}(x),
\end{equation}
where $h(x)$ is the image of $h(n)$ under the Gelfand transform in $C(X_{f})$.
This implies that $\frac{1}{N_{m}}\sum_{n=0}^{N_{m}-1}\delta_{T^{n}x}$ weak* converges to $\nu_{1}$ in the space of Borel probability measures on $X_{f}$. Choose a $\nu$ in the weak* closure of $\{\frac{1}{N_{m}}\sum_{n=0}^{N_{m}-1}\delta_{T^{n}x_{0}}\}_{m=1}^{\infty}$ in the space of Borel probability measures on $X$. When restricted to $X_{f}$, $\nu$ is identified as $\nu_{1}$ by Proposition \ref{a factor}. Then by equation (\ref{formula1 in proposition 1.16}), there is  a sequence $\{n_{j}\}_{j=1}^{\infty}$ of positive integers such that
\[\lim_{j\rightarrow\infty}\int_{X}|f\circ T^{n_{j}}(x)-f(x)|^2d\nu(x)=0.\]
Note that the image of $A^{n_{j}}g(n)$ under the Gelfand transform is $f\circ T^{n_{j}}(x)$ in $C(X_{f})$. Then by equation (\ref{formula2 in proposition 1.16}),
\[\lim_{j\rightarrow\infty}E(|A^{n_{j}}g-g|^2)=0.\]
So $g$ is an asymptotically periodic function. Assume that Problem 1 holds, then
\[\lim_{N\rightarrow \infty}\frac{1}{N}\sum_{n=0}^{N-1}\mu(n)g(n)=\lim_{N\rightarrow \infty}\frac{1}{N}\sum_{n=0}^{N-1}\mu(n)f(T^{n}x_{0})=0.\]

In the remaining part, we prove that Problem 2 implies the disjointness of $\mu$ from all asymptotically periodic function. Suppose that $h(n)$ is an asymptotically periodic function, i.e., for any mean state $E$, there is a sequence $\{n_{j}\}_{j=1}^{\infty}$ of positive integers such that $\lim_{j\rightarrow \infty}\|h-A^{n_{j}}h\|_{E}=0$. Let $(X_{h},\sigma_{A})$ (or $\mathcal{A}_{h}$) be the anqie generated by $h$. Let $x_{0}=\iota(0)$ (corresponding to $(h(0),h(1),\ldots)$)$\in X_{h}$. Suppose that $\frac{1}{N_{m}}\sum_{n=0}^{N_{m}-1}\delta_{(\sigma_{A})^{n}x_{0}}$ weak* converges to a Borel probability measure $\nu$ as $m\rightarrow \infty$. Choose a free ultrafilter $\omega$ in the weak* closure of $\{N_{m}: m=1,2,3,\ldots\}$ in $\beta\mathbb{N}$. Then applying Theorems \ref{the measure induced by a state} to the mean state $E$ depending on $\omega$, we obtain for any $\widetilde{f}(n)\in \mathcal{A}_{h}$,

\[E(\widetilde{f})=\lim_{m\rightarrow \infty}\frac{1}{N_{m}}\sum_{n=0}^{N_{m}-1}\widetilde{f}(n)=\int_{X_{f}}\widetilde{f}(x)d\nu(x),\]
where $\widetilde{f}(x)$ is the image of $\widetilde{f}(n)$ under the Gelfand transform (see equation (\ref{Gelfand transform})).
Then for any $\widetilde{f}(x)\in C(X_{h})$, $\lim_{j\rightarrow \infty}\|\widetilde{f}\circ(\sigma_{A})^{n_{j}}(x)-\widetilde{f}(x)\|_{L^2(\nu)}=0$. So $(X_{h},\sigma_{A},x_{0})$ satisfies the condition in Problem 2. Hence \[\lim_{N\rightarrow \infty}\frac{1}{N}\sum_{n=0}^{N-1}\mu(n)h((\sigma_{A})^{n}x_{0})=\lim_{N\rightarrow \infty}\frac{1}{N}\sum_{n=0}^{N-1}\mu(n)h(n)=0.\]
\end{proof}

\appendix
\section{Mean and large values theorems}
In this section, we list some lemmas that are used in the proof of Lemma \ref{mean value of Dirichlet polynomials for characters}. They are hybrid versions of the corresponding results in \cite{KM}. We refers readers to \cite[Section 3]{KMT} or \cite[Theorems 6.4; 8.3]{Mon} for detailed proofs about Lemmas \ref{square integral large sieve for characters}, \ref{Discrete large values estimate for characters}, and \ref{Discrete large sieve for characters}.
\begin{lemma}\label{square integral large sieve for characters}
Let $T,N,k\geq 1$ and $\{a_{n}\}_{n=1}^{\infty}$ be a sequence of complex numbers. Then
\[\sum_{\chi(mod~k)}\int_{0}^{T}|\sum_{n\leq N}a_{n}\chi(n)n^{it}|^2dt\ll (\varphi(k)T+\frac{\varphi(k)}{k}N)\sum_{\substack{n\leq N\\(n,k)=1}}|a_{n}|^2\]
\end{lemma}

\begin{lemma}\label{Discrete large values estimate for characters} Let $T, N,k\geq 1$ and $\{a_{n}\}$ be any complex numbers. Let $\mathcal{E}$ be a subset of  $\{\chi(mod~k)\}\times [-T,T]$ satisfying that $|t-u|\geq 1$ whenever $(\chi,t),(\chi,u)\in \mathcal{E}$ with $t\neq u$. Then
\[\sum_{(\chi,t)\in \mathcal{E}}|\sum_{n\leq N}a_{n}\chi(n)n^{it}|^2\ll \Big(\varphi(k)T+\frac{\varphi(k)}{k}N\Big)\log(3k)\sum_{\substack{n\leq N\\(n,k)=1}}|a_{n}|^2.\]
\end{lemma}
Applying the above lemma with an argument similar to the proof of \cite[Lemma 8]{KM}, we have the following.
\begin{lemma}\label{Basic large values estimate-prime support}
Let $P,T\geq 2$, $k\geq 1$ and $V>0$. Write $$P_{\chi}(s)=\ds\sum_{P\leq p\leq 2P}\frac{a_{p}\chi(p)}{p^s}$$ with $|a_{p}|\leq 1$ for $p\leq 2P$. Let $\mathcal{R}(\mathcal{T},V)$ be a subset of $\{(\chi,t)\in \{\chi (mod~k)\}\times [-T,T]:P_{\chi}(1+it)\geq V^{-1}\}$ satisfying $|t-u|\geq 1$ whenever $(\chi,t),(\chi,u)\in \mathcal{R}(\mathcal{T},V)$ with $t\neq u$. Then
$$\#\mathcal{R}(\mathcal{T},V)\ll (kT)^{2\frac{\log V}{\log P}}V^2\exp\Big(2\frac{\log (kT)}{\log P}\log\log (kT)\Big).$$
\end{lemma}
The following is a hybrid version of ``Hal\'{a}sz inequality for integers" stated in \cite[Lemma 9]{KM}.
\begin{lemma}\label{Discrete large sieve for characters} With the same assumptions as Lemma \ref{Discrete large values estimate for characters}. We have
\[\sum_{(\chi,t)\in \mathcal{E}}|\sum_{n\leq N}a_{n}\chi(n)n^{it}|^2\ll \Big(\frac{\varphi(k)}{k}N+|\mathcal{E}|(kT)^{\frac{1}{2}}\log(2kT)\Big)\sum_{\substack{n\leq N\\(n,k)=1}}|a_{n}|^2.\]
\end{lemma}
When $a_{n}$ is supported on the set of primes, we have the following hybrid version of ``Hal\'{a}sz inequality for primes" stated in \cite[Lemma 11]{KM}.
\begin{lemma}\label{hybrid version of Halasz inequality for primes} Let $P,T\geq 2$ and $k<(\log P)^{\frac{4}{3}-\epsilon}$. Let $\mathcal{E}$ be a subset of  $\{\chi(mod~k)\}\times [-T,T]$ satisfying that $|t-u|\geq 1$ whenever $(\chi,t),(\chi,u)\in \mathcal{E}$ with $t\neq u$. Then
\begin{align*}
  &\sum_{(\chi,t)\in \mathcal{E}}|\sum_{P\leq p\leq 2P}a_{p}\chi(p)p^{it}|^2\ll \Big(\varphi(k)P+|\mathcal{E}|P\exp(-\frac{\log P}{(\log (P+T))^{\frac{2}{3}+\epsilon}})(\log (P+T))^5
  \Big)\sum_{P\leq p\leq 2P}\frac{|a_{p}|^2}{\log P},
\end{align*}
where $\epsilon$ is a sufficiently small positive number.
\end{lemma}
\begin{proof}
By the duality principle applied to $(\chi(p)p^{it})_{P\leq p\leq 2P, (\chi,t)\in \mathcal{E}}$, it is enough to prove that for any complex numbers $\eta_{\chi,t}$,
\begin{align*}
  \sum_{P\leq p\leq 2P}\log p\big|\sum_{(\chi,t)\in \mathcal{E}}\eta_{\chi,t}\chi(p)p^{it}\big|^2\ll &\Big(|\mathcal{E}|P\exp(-\frac{\log P}{(\log (P+T))^{\frac{2}{3}+\epsilon}})(\log (P+T))^5\Big)\\
  &+\varphi(k)P\Big)\sum_{(\chi,t)\in \mathcal{E}}|\eta_{\chi,t}|^2.
\end{align*}
Let $f(x)$ be a smooth compactly supported function on $[1/2,5/2]$ such that $f(x)=1$ for $1\leq x\leq 2$ and $f$ decays to zero outside of the interval $[1,2]$. Let $\widetilde{f}$ denote the Mellin transform of $f$. Then $\widetilde{f}(x+iy)\ll_{A}(1+|y|^{-2})$ uniformly in $|x|\leq A$.
Then
\begin{align*}
 &\sum_{P\leq p\leq 2P}\log p\big|\sum_{(\chi,t)\in \mathcal{E}}\eta_{\chi,t}\chi(p)p^{it}\big|^2\leq \sum_{p^{l}}\log p\big|\sum_{(\chi,t)\in \mathcal{E}}\eta_{\chi,t}\chi(p^{l})p^{ilt}\big|^2f(\frac{p^{l}}{P})\\
  &\leq \sum_{(\chi,t),(\chi,t_{1})\in \mathcal{E}}|\eta_{\chi,t}\eta_{\chi,t_{1}}|\sum_{p^{l}}(\log p) p^{il(t-t_{1})}\chi_{0}(p^{l})f(\frac{p^{l}}{P})|\\
  &+\sum_{\substack{(\chi,t),(\chi_{1},t_{1})\in \mathcal{E}\\ \chi_{1}\neq \chi}}|\eta_{\chi,t}\eta_{\chi_{1},t_{1}}|\sum_{p^{l}}(\log p) p^{il(t-t_{1})}\chi(p^{l})\overline{\chi_{1}}(p^{l})f(\frac{p^{l}}{P})|.
\end{align*}
When $\chi$ is not a principal character modulo $k$, Perron's formula with the zero-free region for $L(s,\chi)$ gives for $|\alpha|\leq T$,
\[\sum_{P<p<2P}p^{i\alpha}\chi(p)\ll P\exp(-\frac{\log P}{(\log (P+T))^{\frac{2}{3}+\epsilon}})(\log (P+T))^4.\]
Combining with $ab\leq \frac{a^2+b^2}{2}$, we have
\begin{align*}
 &\sum_{\substack{(\chi,t),(\chi_{1},t_{1})\in \mathcal{E}\\ \chi_{1}\neq \chi}}|\eta_{\chi,t}\eta_{\chi_{1},t_{1}}|\sum_{p^{l}}(\log p) p^{il(t-t_{1})}\chi(p^{l})\overline{\chi_{1}}(p^{l})f(\frac{p^{l}}{P})|
 \\
 &\ll \sum_{\substack{(\chi,t),(\chi_{1},t_{1})\in \mathcal{E}\\ \chi_{1}\neq \chi}}(|\eta_{\chi,t}|^2+|\eta_{\chi_{1},t_{1}}|^2)P\exp(-\frac{\log P}{(\log (P+T))^{\frac{2}{3}+\epsilon}})(\log (P+T))^5\\
 &\ll|\mathcal{E}|P\exp(-\frac{\log P}{(\log (P+T))^{\frac{2}{3}+\epsilon}})(\log (P+T))^5\sum_{(\chi,t)\in \mathcal{E}}|\eta_{\chi,t}|^2.
\end{align*}
It follows from a similar argument to the proof of Lemma 11 in \cite{KM} that

\begin{align*}
 &\sum_{(\chi,t),(\chi,t_{1})\in \mathcal{E}}|\eta_{\chi,t}\eta_{\chi,t_{1}}|\Big|\sum_{p^{l}}(\log p) p^{il(t-t_{1})}\chi_{0}(p^{l})f(\frac{p^{l}}{P})\Big|
 \\
 &\ll \sum_{(\chi,t),(\chi,t_{1})\in \mathcal{E}}(|\eta_{\chi,t}|^2+|\eta_{\chi,t_{1}}|^2)\Big(\Big|\sum_{p^{l}}(\log p) p^{il(t-t_{1})}f(\frac{p^{l}}{P})\Big|+\log P\log k\Big)\\
 &\ll(\varphi(k)P+|\mathcal{E}|P\exp(-\frac{\log P}{(\log T)^{\frac{2}{3}+\epsilon}})(\log T)^2+|\mathcal{E}|\log k\log P)\sum_{(\chi,t)\in \mathcal{E}}|\eta_{\chi,t}|^2.
\end{align*}
\end{proof}
The proofs of the next two lemmas are almost the same as the proofs of Lemmas 12, 13 in \cite{KM} with the following small differences: instead of the standard mean value theorem for Dirichlet polynomials, we apply Lemma \ref{square integral large sieve for characters}; one obtains the extra factor $\varphi(k)/k$ due to the coefficients are supported on the integers $(n,k)=1$.
\begin{lemma}\label{fractorization}
Let $X,H\geq 1$ and $Q>P\geq 2$. Suppose that $a_{mp}=b_{m}c_{p}$, $p\nmid m, P\leq p\leq  Q$, where the sequences $\{a_{m}\}_{m},\{b_{m}\}_{m},\{c_{p}\}_{p}$ are bounded. Let $k\geq 1$ and $\mathcal{M}$ be a collection of Dirichlet characters modulo $k$.
Let $$Q_{v,H}(\chi,s)=\ds\sum_{\substack{P\leq p\leq Q\\ e^{\frac{v}{H}}\leq p\leq e^{\frac{v+1}{H}}}}\frac{c_{p}\chi(p)}{p^s}$$ and
$$R_{v,H}(\chi,s)=\ds\sum_{Xe^{-\frac{v}{H}}\leq 2Xe^{-\frac{v}{H}}}\frac{b_{m}\chi(m)}{m^s}\frac{1}{\#\{P\leq q\leq Q:q|m,q~ is~ a ~prime\}+1}.$$
Let $\mathcal{T}_{\chi}\subseteq [-T,T]$, and $\mathcal{I}=\{j\in \mathbb{N}:\lfloor H\log P\rfloor \leq j\leq H\log Q\}$. Then
\begin{eqnarray*}
  \sum_{\chi\in \mathcal{M}}\int_{\mathcal{T}_{\chi}}|\sum_{X\leq m\leq 2X}\frac{a_{m}\chi(m)}{m^{1+it}}|^2dt \ll H\log(\frac{Q}{P})\times \sum_{\chi \in \mathcal{M}}\sum_{j\in \mathcal{I}}\int_{\mathcal{T}_{\chi}}|Q_{j,H}(\chi,1+it)R_{j,H}(\chi,1+it)|^2dt&& \\
   +\frac{\varphi(k)}{k}\frac{\varphi(k)T+{(\varphi(k)/k})X}{X}\Big(\frac{1}{H}+\frac{1}{P}\Big)+\frac{\varphi(k)T+
   {(\varphi(k)/k)}X}{X}
   \sum_{\substack{X\leq m\leq 2X\\(m,k\prod_{P\leq p\leq Q}p)=1}}\frac{|a_{m}|^2}{m}. &&
\end{eqnarray*}
\end{lemma}
\begin{lemma}\label{Moment computation}
Let $k,T\geq 1$, $Y_{2}\geq Y_{1}\geq 2$ and $l=\lceil\frac{\log Y_{2}}{\log Y_{1}}\rceil$.  Let $\{a_{m}\}_{m}$ and $\{c_{p}\}_{p}$ be bounded sequences. Suppose that $X$ is sufficiently large. Let
$$Q(\chi,s)=\sum_{Y_{1}\leq p\leq 2Y_{1}}\frac{c_{p}\chi(p)}{p^s}$$ and $$R(\chi,s)=\sum_{X/Y_{2}\leq m\leq 2X/Y_{2}}\frac{a_{m}\chi(m)}{m^{s}}.$$
Then
$$\sum_{\chi(mod~k)}\int_{-T}^{T}|Q(\chi,1+it)^{l}R(\chi,1+it)|^2dt\ll \frac{\varphi(k)}{k}(\varphi(k)\frac{T}{X}+\frac{\varphi(k)}{k}2^{l}Y_{1})(l+1)!^2.$$
\end{lemma}
The following Parseval bound follows exactly in the same way as \cite[Lemma 14]{KM} with no need to consider the difference of two averages as the integral function.
\begin{lemma}\label{Parseval bound}
Suppose that $\{a_{m}\}_{m=1}^{\infty}$ be a bounded sequence. Assume that $X\geq 2$ and $1\leq h\leq X$. Write
\[A(s):=\sum_{X\leq m\leq 4X}\frac{a_{m}}{m^{s}}.\]
Then
\begin{equation}\label{parseval bound}
\frac{1}{X}\int_{X}^{2X}|\frac{1}{h}\sum_{x\leq n\leq x+h}a_{n}|^2dx\ll \int_{1}^{1+iX/h}|A(s)|^2|ds|+\max_{T\geq X/h}\frac{X/h}{T}\int_{1+iT}^{1+2iT}|A(s)|^2|ds|.
\end{equation}
\end{lemma}
\section{Lemmas on multiplicative functions}

In this section, we give some lemmas on the pointwise bounds of Dirichlet polynomials with coefficients supported on integers coprime to a fixed number. We start from the following lemma which has almost identical proof to that of \cite[Corollary 2.2]{BGS} with the small modification: one applies the refinement of the Hal\'{a}sz-Montgometry-Tenenbaum result (\cite[Corollary 1]{GS}), rather than the Hal\'{a}sz inequality. This leads to that the bound $O(\frac{1}{\sqrt{T}})$ is improved by $O(\frac{1}{T})$.
\begin{lemma}\label{similar corollary in GS}
Let $x\geq 3$, $1\leq k\leq x$ and $1\leq T\leq (\log x)^{\frac{1}{4}}$. Let $f(n)$ be a multiplicative function with $|f(n)|\leq 1$ for all $n\in \mathbb{N}$. Then
\[\frac{1}{x}\sum_{\substack{n\leq x\\(n,k)=1}}f(n)\ll \frac{\varphi(k)}{k}\Big((M_{k}(f;x;T)+1)\exp(-M_{k}(f;x;T))+\frac{1}{T}\Big).\]
\end{lemma}
While for large $T$ in the above lemma, it follows directly from \cite[Lemma 2.2]{KMT} that
\begin{lemma}\label{similar lemma in KMT}
Let $x\geq 3$, $1\leq k\leq x$ and $(\log x)^{\frac{1}{4}}<T\leq x$. Let $f(n)$ be a multiplicative function with $|f(n)|\leq 1$ for all $n\in \mathbb{N}$. Then
\[\frac{1}{x}\sum_{\substack{n\leq x\\(n,k)=1}}f(n)\ll \frac{\varphi(k)}{k}\Big((M_{k}(f;x;T)+1)\exp(-M_{k}(f;x;T))+(\log x)^{-\frac{5}{64}}\Big).\]
\end{lemma}
Combining with Lemmas \ref{similar corollary in GS} and \ref{similar lemma in KMT}, we have the following Hal\'{a}sz-type inequality for the mean values of multiplicative functions.
\begin{lemma}\label{a halasz type inequality}
Let $x\geq 3$ and $1\leq k,T\leq x$. Let $f(n)$ be a multiplicative function with $|f(n)|\leq 1$ for all $n\in \mathbb{N}$. Then
\[\frac{1}{x}\sum_{\substack{n\leq x\\(n,k)=1}}f(n)\ll \frac{\varphi(k)}{k}\Big((M_{k}(f;x;T)+1)\exp(-M_{k}(f;x;T))+\frac{1}{T}+(\log x)^{-\frac{5}{64}}\Big).\]
\end{lemma}
The following lemma follows immediately from Lemma \ref{a halasz type inequality} and partial summation.

\begin{lemma}\label{the point bound of Dirichlet polynomial}
Let $x\geq 3$ and $1\leq k,T_{0}\leq x$. Suppose that $\chi$ is a Dirichlet character modulo $k$. Let $f(n)$ be a multiplicative function with $|f(n)|\leq 1$, and let
\[F(\chi,s)=\sum_{x\leq n\leq 2x}\frac{f(n)\chi(n)}{n^{s}}.\]
Let
\begin{equation}\label{def of L}
L(f\chi; x;T_{0})=\inf_{|t_{0}|\leq T_{0}}\mathbb{D}_{k}(f\chi,n\mapsto n^{it+it_{0}};x)^2.
\end{equation}
Then
\[|F(\chi,\sigma+it)|\ll x^{1-\sigma}\frac{\varphi(k)}{k}\Big((L(f\chi;x;T_{0})+1)\exp(-L(f\chi; x;T_{0}))+\frac{1}{T_{0}}+(\log x)^{-\frac{5}{64}}\Big).\]
\end{lemma}
Actually, in the proof of Theorem \ref{mean value of Dirichlet polynomials for characters}, we also need to apply the Hal\'{a}sz-type inequality to a Dirichlet polynomial of the form $F_{v,H}(\chi,s)$ in Lemma \ref{fractorization} with the coefficients not quite multiplicative. Using Lemma \ref{the point bound of Dirichlet polynomial}, a similar argument to the proof of Lemma 3 in \cite{KM} gives the following result.
\begin{proposition}\label{a Halasz bound}
Let $X\geq Q> P\geq 2$. Let $1\leq k,T_{0}\leq X$ and $\chi$ be a Dirichlet character modulo $k$. Let $f(n)$ be a multiplicative function with $|f(n)|\leq 1$ and
\[R(\chi,s)=\sum_{X\leq n\leq 2X}\frac{f(n)\chi(n)}{n^{s}}\frac{1}{\#\{P\leq q\leq Q:q|m,q~ is~ a ~prime\}+1}.\]
Suppose that $\delta(n)$ is the characteristic function supported on the set of all integers between $1$ and $2X$ which is coprime to $\prod_{P\leq p\leq Q}p$.
Then for any $t$,
\begin{align*}
  |R(\chi,1+it)|\ll &\frac{\log Q}{\log P}\frac{\varphi(k)}{k}\Big((L(\delta f\chi;X;T_{0})+1)\exp(-L(\delta f\chi; X;T_{0}))+\frac{1}{T_{0}}+(\log x)^{-\frac{5}{64}}\Big)  \\
  &+(\log X)\exp(-\frac{\log X}{3\log Q}\log \frac{\log X}{\log Q}),
\end{align*}
where $L(\delta f\chi;X;T_{0})$ is defined as equation (\ref{def of L}).
\end{proposition}
\section{Proof of Proposition \ref{main theorem2}}
In this section we shall first prove Proposition \ref{main theorem2}, which states that the average of a 1-bounded multiplicative function is small for almost all short arithmetic progressions when it it not $\chi(p)p^{it}$ pretentious. The proof of this result can be reduced to proving the following lemma.
\begin{lemma}\label{mean value of Dirichlet polynomials for characters}
Let $X$ be large enough such that $1\leq k\leq (\log X)^{1/32}$. Suppose that $2\leq h\leq X/k$. Let $f(n)$ be a multiplicative function with $|f(n)|\leq 1$ for all $n\geq 1$, and let \[F(\chi,s)=\sum_{X\leq n\leq 2 X}\frac{f(n)\chi(n)}{n^{s}}.\]
Then, for any $T\geq 1$,
\begin{align*}
\sum_{\chi(mod~k)}\int_{0}^{T}|F(\chi,1+it)|^2dt  \ll & \frac{\varphi(k)}{k}(\frac{\varphi(k)T}{X/h}+\frac{\varphi(k)}{k})\Big(\frac{k}{\varphi(k)}\frac{\log\log h}{\log h}+\frac{1}{(\log X)^{1/300}}\Big)\\
  &+\frac{\varphi^2(k)}{k^2}\Big(M_{k}(f;k;X;2X)+1)\exp(-M_{k}(f;k;X;2X)\Big).
\end{align*}
\end{lemma}
Some results used below are given in Appendices A and B. We first show that the above lemma implies Proposition \ref{main theorem2}.
\begin{proof}[Proof of Proposition \ref{main theorem2} (Assume that Lemma \ref{mean value of Dirichlet polynomials for characters} holds)]
By the Parseval bound stated in formula (\ref{parseval bound}) and Lemma \ref{mean value of Dirichlet polynomials for characters},
\begin{align*}
& \frac{1}{k^2h^2X}\sum_{\chi(mod~k)}\int_{X}^{2X}\Big|\sum_{n=x}^{x+hk}f(n)\chi(n)\Big|^2dx \\
\ll& \sum_{\chi(mod~k)}\int_{1}^{1+i\frac{X}{kh}}\Big|\sum_{X\leq m\leq 4X}\frac{f(m)\chi(m)}{m^s}\Big|^2|ds|\\
&+\max_{T\geq \frac{X}{kh}}\frac{X/kh}{T}\sum_{\chi(mod~k)}\int_{1+iT}^{1+2iT}\Big|\sum_{X\leq m\leq 4X}\frac{f(m)\chi(m)}{m^s}\Big|^2|ds|\\
\ll& \frac{\varphi^2(k)}{k^2}\Big(\frac{k}{\varphi(k)}\frac{\log \log h}{\log h}+\frac{1}{(\log X)^{1/300}}\Big)
  +\frac{\varphi^2(k)}{k^2}\frac{M_{k}(f;k;X;2X)+1}{\exp(M_{k}(f;k;X;2X))}.
\end{align*}
Hence
\begin{align*}
 &\sum_{\substack{a=1\\(a,k)=1}}^{k}\sum_{x=X}^{2X}\Big|\sum_{\substack{n=x\\n\equiv a(mod~k)}}^{x+hk}f(n)\Big|^2 \\
 =& \frac{1}{\varphi^2(k)}\sum_{\substack{a=1\\(a,k)=1}}^{k}\sum_{x=X}^{2X}\Big|\sum_{\chi(mod~k)}\overline{\chi}(a)\sum_{n=x}^{x+hk}f(n)\chi(n)\Big|^2\\
 =& \frac{1}{\varphi^2(k)}\sum_{x=X}^{2X}\sum_{\chi_{1},\chi_{2}(mod~k)}\Big(\sum_{\substack{a=1\\(a,k)=1}}^{k}\overline{\chi_{1}}(a)\chi_{2}(a)\Big)
  \Big(
  \sum_{n=x}^{x+hk}f(n)\chi_{1}(n)\Big)\Big(
  \sum_{n=x}^{x+hk}\overline{f}(n)\overline{\chi_{2}}(n)\Big)\\
 =& \frac{1}{\varphi(k)}\sum_{x=X}^{2X}\sum_{\chi(mod~k)}\Big|
  \sum_{n=x}^{x+hk}f(n)\chi(n)\Big|^2
  = \frac{1}{\varphi(k)}\sum_{\chi(mod~k)}\int_{X}^{2X}\Big|
  \sum_{n=x}^{x+hk}f(n)\chi(n)\Big|^2dx\\
  \ll&  h^2X\varphi(k)\Big(\frac{k}{\varphi(k)}\frac{\log \log h}{\log h}+\frac{1}{(\log X)^{1/300}}+\frac{M_{k}(f;k;X;2X)+1}{\exp(M_{k}(f;k;X;2X))}\Big).
\end{align*}
\end{proof}
Now we start to prove Lemma \ref{mean value of Dirichlet polynomials for characters}.
\begin{proof}[Proof of Lemma \ref{mean value of Dirichlet polynomials for characters}]
Since the hybrid mean value theorem (see Lemma \ref{square integral large sieve for characters}) gives the bound $O\Big(\frac{\varphi(k)}{k}(\frac{\varphi(k)T}{X}+\frac{\varphi(k)}{k})\Big)$, we can assume that $T\leq X$.
Let $\chi_{1}$ be the character modulo $k$ minimizing the distance $\inf_{|t|\leq 2X}\mathbb{D}_{k}(f\chi,n\mapsto n^{it};X)$. Let $t_{1}$ be the real number minimizing $\mathbb{D}_{k}(f\chi_{1},n\mapsto n^{it};X)$. Then for any $\chi(mod~k)$ and $|t|\leq 2X$,  $\mathbb{D}_{k}(f\chi,n\mapsto n^{it};X)\geq \mathbb{D}_{k}(f\chi_{1},n\mapsto n^{it_{1}};X)$. Next we claim that
for $\chi\neq \chi_{1}$ and any $t$ with $|t|\leq 2X$,
\begin{equation}\label{nonprincipal character estimate}
2\mathbb{D}_{k}(f\chi,n\mapsto n^{it};X)\geq(\frac{1}{\sqrt{3}}-\epsilon)\sqrt{\log\log X}+O(1)
\end{equation}
and for $\chi=\chi_{1}$ and $|t-t_{1}|\geq 1$,
\begin{equation}\label{principal character estimate}
2\mathbb{D}_{k}(f\chi_{1},n\mapsto n^{it};X)\geq(\frac{1}{\sqrt{3}}-\epsilon)\sqrt{\log\log X}+O(1),
\end{equation}
where $\epsilon>0$ is sufficiently small.
In fact, suppose first that $f$ is unimodular, i.e., $|f(n)|=1$ for all $n\geq 1$. By the triangle inequality of $\mathbb{D}_{k}$ (see, e.g., \cite[Lemma 3.1]{BGS07}),
\begin{align*}
 2\mathbb{D}_{k}(f\chi,n\mapsto n^{it};X)&\geq \mathbb{D}_{k}(f\chi,n\mapsto n^{it};X)+\mathbb{D}_{k}(f\chi_{1},n\mapsto n^{it_{1}};X) \\
  &=\mathbb{D}_{k}(\overline{f};n\mapsto \chi(n) n^{-it};X)+\mathbb{D}_{k}(f,n\mapsto\overline{\chi_{1}}(n) n^{it_{1}};X)\\
  &\geq \mathbb{D}_{k}(\overline{f}f, n\mapsto\overline{\chi_{1}}\chi(n) n^{i(t_{1}-t)};X)=\mathbb{D}_{k}(1, n\mapsto\overline{\chi_{1}}\chi(n) n^{i(t_{1}-t)};X).
\end{align*}
If $f$ is not unimodular, by means of the method used in \cite[Lemma 2.2]{KMT}, we can model $f$ by a stochastic multiplicative function $\textbf{f}$ such that $\{\textbf{f}(n)\}_{n}$ being a sequence of unimodular random variables defined on certain probability space, and for each prime $p$ the expectation $\mathbb{E}\textbf{f}(p)=f(p)$. By linearity of the expectation, we thus have
\[\mathbb{D}_{k}(f\chi,n\mapsto n^{it};X)^2=\sum_{\substack{p\leq x\\ p\nmid k}}\frac{1-\text{Re}(p^{-it}\chi(p)\mathbb{E}\textbf{f}(p))}{p}=\mathbb{E}\Big(\mathbb{D}_{k}(\textbf{f}\chi,n\mapsto n^{it};X)^2\Big).\]
Since $\textbf{f}$ is unimodular, $2\mathbb{D}_{k}(\textbf{f}\chi,n\mapsto n^{it};X)\geq \mathbb{D}_{k}(1, n\mapsto\overline{\chi_{1}}\chi(n) n^{i(t_{1}-t)};X)$. Hence formulas (\ref{nonprincipal character estimate}) and (\ref{principal character estimate}) hold.

Write $[0,T]=\mathcal{L}_{1}\cup\mathcal{L}_{2}$, where
\[\mathcal{L}_{1}=\{0\leq t\leq T: |t-t_{1}|< (\log X)^{\frac{5}{64}}\},\]
\[\mathcal{L}_{2}=\{0\leq t\leq T: |t-t_{1}|\geq (\log X)^{\frac{5}{64}}\}.\]
We now first estimate $\sum_{\chi(mod~k)}\int_{\mathcal{L}_{2}}|F(\chi,1+it)|^2dt$. By means of similar ideas in the proof of \cite[Proposition 1]{KM}, we first split the integral over $\mathcal{L}_{2}$ into several parts according to the typical factorization when $n$ is restricted to a dense subset $\mathcal{S}\subseteq [X,2X]$. Recall that $\mathcal{S}$ in \cite{KM} is defined to be the set of all integers $X\leq n\leq 2X$ having at least one prime factor in each interval $[P_{j},Q_{j}]$ for $j\leq J$, where $J$ is chosen to be the largest index $j$ such that $Q_{j}\leq \exp((\log X)^{\frac{1}{2}})$. The choice of $P_{j},Q_{j}$ needs to satisfy some requirements as in \cite{KM}.
Now we set the same parameters $\alpha_{j}:=\frac{1}{4}-\eta(1+\frac{1}{2j})$, $\eta:=1/150$, $H_{j}:=j^2 P_{1}^{\frac{1}{6}-\eta}/(\log Q_{1})^{\frac{1}{3}}$, $\mathcal{I}_{j}:=[v\in \mathbb{N}:\lfloor H_{j}\log P_{j}\rfloor\leq v\leq H_{j}\log  Q_{j}]$ as in \cite{KM}.
Define for $v\in \mathcal{I}_{j}$,
$$R_{v,H_{j}}(\chi,1+it):=\sum_{Xe^{-v/H_{j}}\leq m\leq 2Xe^{-v/H_{j}}}\frac{f(m)\chi(m)}{m^s}\frac{1}{\sharp\{P_{j}\leq p\leq Q_{j}:p|m\}+1}$$
and
$$Q_{v,H_{j}}(\chi,s):=\ds\sum_{\substack{P_{j}\leq q\leq Q_{j}\\ e^{v/H_{j}}\leq q\leq e^{(v+1)/H_{j}}}}\frac{f(q)\chi(q)}{q^s}.$$
Let $\mathcal{T}_{j}$ denote the set of all $(\chi,t)\in \{\chi(mod~k)\}\times \mathcal{L}_{2}$ with $j$ the smallest index such that for all $v\in \mathcal{I}_{j}$, $|Q_{v,H_{j}}(\chi,1+it)|\leq e^{-\alpha_{j}v/H_{j}}$. Let $\mathcal{U}$ be the complement of union of $\mathcal{T}_{j}$. We may also write that for some sets $\mathcal{T}_{j,\chi}$, $\mathcal{U}_{\chi}\subseteq \mathcal{L}_{2}$, $\mathcal{T}_{j}=\bigcup_{\chi(mod~k)}\{\chi\}\times \mathcal{T}_{j,\chi}$ and $\mathcal{U}=\bigcup_{\chi(mod~k)}\{\chi\}\times \mathcal{U}_{\chi}$. Then
\begin{equation}\label{split the integral into several subintegrals}
\sum_{\chi(mod~k)}\int_{\mathcal{L}_{2}}|F(\chi,1+it)|^2dt=\sum_{j=1}^{J}\sum_{\chi(mod~k)}\int_{\mathcal{T}_{j,\chi}}|F(\chi,1+it)|^2dt+\sum_{\chi(mod~k)}
\int_{\mathcal{U}_{\chi}}|F(\chi,1+it)|^2dt.
\end{equation}
By the the fundamental lemma of the sieve,
\[\sum_{\substack{X\leq m\leq 2X\\(m,k\prod_{P_{j}\leq p\leq Q_{j}}p)=1}}1\ll X\frac{\varphi(k)}{k}\frac{\log P_{j}}{\log Q_{j}}\prod_{\substack{P_{j}\leq p\leq Q_{j}\\p|k}}(1-\frac{1}{p})^{-1}\leq  X\frac{\varphi(k)}{k}\frac{\log P_{j}}{\log Q_{j}}\frac{k}{\varphi(k)}.\]
Using Lemma \ref{fractorization} with $H=H_{j},P=P_{j},Q=Q_{j}$ and $a_{m}=b_{m}=f(m)\chi(m),c_{p}=f(p)\chi(p)$ and the above inequality,
we obtain
\begin{align*}
  \sum_{\chi(mod~k)}\int_{\mathcal{T}_{j,\chi}}|F(\chi,1+it)|^2dt & \ll H_{j}\log(\frac{Q_{j}}{P_{j}})\sum_{\chi(mod~k)}\sum_{v\in \mathcal{I}_{j}}\int_{\mathcal{T}_{j,\chi}}|Q_{v,H}(\chi,1+it)R_{v,H}(\chi, 1+it)|^2dt \\
  &+\frac{\varphi(k)}{k}\frac{\varphi(k)T+(\varphi(k)/k)X}{X}\Big(\frac{1}{H_{j}}+\frac{1}{P_{j}}+\frac{k}{\varphi(k)}\frac{\log P_{j}}{\log Q_{j}}\Big)
\end{align*}
Here the second term contributes totally to the right-hand side of formula (\ref{split the integral into several subintegrals}),
\begin{align}\label{eq111}
\begin{aligned}
&\frac{\varphi(k)}{k}\frac{\varphi(k)T+(\varphi(k)/k)X}{X}\sum_{j=1}^{J} \Big(\frac{1}{H_{j}}+\frac{1}{P_{j}}+\frac{k}{\varphi(k)}\frac{\log P_{j}}{\log Q_{j}} \Big) \\
\ll & \frac{\varphi(k)}{k}\frac{\varphi(k)T+(\varphi(k)/k)X}{X}\Big(\frac{(\log Q_{1})^{\frac{1}{3}}}{P_{1}^{\frac{1}{6}-\eta}} +\sum_{j=1}^{J}\frac{1}{P_{1}^{j^2}}
+\frac{k}{\varphi(k)}\frac{\log P_{1}}{\log Q_{1}}\Big) \\
\ll& \frac{\varphi(k)}{k}\frac{\varphi(k)T+(\varphi(k)/k)X}{X}
\Big(\frac{(\log Q_{1})^{\frac{1}{3}}}{P_{1}^{\frac{1}{6}-\eta}}+\frac{k}{\varphi(k)}\frac{\log P_{1}}{\log Q_{1}} \Big).
\end{aligned}
\end{align}
In the above, we use the relation that $H_{j}=j^2P_{1}^{\frac{1}{6}-\eta}/(\log Q_{1})^{\frac{1}{3}}$ and $\log P_{j}\geq 8j^2/\eta\log Q_{j-1}+16j^2/\eta\log j$.

Now for $1\leq j\leq J$, we focus on bounding
\[E_{j}:=H_{j}\log Q_{j}\sum_{\chi(mod~k)}\sum_{v\in \mathcal{\mathcal{I}}_{j}}\int_{\mathcal{T}_{j,\chi}}|Q_{v,H_{j}}(\chi,1+it)R_{v,H_{j}}(\chi,1+it)|^2dt.\]
\textbf{Estimate of $E_{1}$.} We repeat the argument in \cite[Section 8.1]{KM} with the difference that the standard mean-value theorem is replaced by the ``hybrid mean-value theorem" (Lemma \ref{square integral large sieve for characters}),
\begin{equation}\label{contribution from E1}
E_{1}\ll (\frac{\varphi(k)T}{X/Q_{1}}+\frac{\varphi(k)}{k})\frac{(\log Q_{1})^{\frac{1}{3}}}{P_{1}^{\frac{1}{6}-\eta}}\frac{\varphi(k)}{k}.
\end{equation}
\textbf{Estimate of $E_{j}$~for~$2\leq j\leq J$.} Let $\mathcal{T}_{j,\chi}^{r}=\{t\in \mathcal{T}_{j,\chi},|Q_{r,H_{j-1}}(\chi,1+it)|> e^{-\frac{r\alpha_{j-1}}{H_{j-1}}}\}$ for $r\in \mathcal{I}_{j-1}$. Then $\mathcal{T}_{j,\chi}=\bigcup_{r\in \mathcal{I}_{j-1}}\mathcal{T}_{j,\chi}^{r}.$
If $\mathcal{T}_{j,\chi}^{r}=\emptyset$, we set $\int_{T_{j,\chi}^{r}}|Q_{r,H_{j-1}}(\chi, 1+it)|^2dt=0$.
Then $$E_{j}\ll H_{j}\log Q_{j}\sum_{v\in \mathcal{I}_{j}}\sum_{r\in \mathcal{I}_{j-1}}\sum_{\chi(mod~k)}e^{-2\frac{\alpha_{j}v}{H_{j}}}\ds\int_{\mathcal{T}_{j,\chi}^{r}}|R_{v,H_{j}}(\chi,1+it)|^2dt.$$
By an argument similar to \cite[Section 8.2]{KM} and Lemmas \ref{square integral large sieve for characters}, \ref{Moment computation}, we obtain
\begin{equation}\label{contribution from Ej}
E_{j}\ll \frac{\varphi(k)}{k}(\frac{\varphi(k)T}{X}+\frac{\varphi(k)}{k})\frac{1}{j^2P_{1}}.
\end{equation}
\textbf{Estimate of $\sum_{\chi(mod~k)}\int_{\mathcal{U}_{\chi}}|F(\chi,1+it)|^2dt$}.
Let $P=\exp((\log X)^{\frac{63}{64}}),Q=\exp(\frac{\log X}{\log\log X}),H=(\log X)^{\frac{1}{64}}$.
Set $\mathcal{I}=[\lfloor H\log P\rfloor, H\log  Q]$.
For $v\in \mathcal{I}$, write
\[Q_{v,H}(\chi,s)=\sum_{\substack{P\leq p\leq Q\\ e^{v/H}\leq p \leq e^{(v+1)/H}}}\frac{f(p)\chi(p)}{p^s},\]
and
\[R_{v,H}(\chi,s)=\sum_{Xe^{-v/H}\leq n \leq 2Xe^{-v/H}}\frac{f(n)\chi(n)}{n^s}\frac{1}{\sharp\{p\in [P,Q]:p|n\}+1}.\]
Note that $k< \log X$ and then $(k,\prod_{P\leq p\leq Q}p)=1$. Applying Lemma \ref{fractorization} with $a_{m}=b_{m}=f(m)\chi(m),c_{p}=f(p)\chi(p)$, we have that for some $v_{0}\in \mathcal{I}$,
\begin{align*}
  &\sum_{\chi(mod~k)}\int_{\mathcal{U}_{\chi}}\Big|\sum_{X\leq m\leq 2X}\frac{f(m)\chi(m)}{m^{1+it}}\Big|^2dt \\
  \ll&  H^2\log^2(\frac{Q}{P}) \sum_{\chi(mod~k)}\int_{\mathcal{U}_{\chi}}|Q_{v_{0},H}(\chi,1+it)R_{v_{0},H}(\chi,1+it)|^2dt\\
 +& \frac{\varphi(k)}{k}\frac{\varphi(k)T+(\varphi(k)/k)X}{X}\Big(\frac{1}{H}+\frac{1}{P}\Big)+\frac{\varphi(k)T+(\varphi(k)/k)X}{X}\frac{\log P}{\log Q}\frac{\varphi(k)}{k}.
\end{align*}
Recall that $\mathcal{W}\subseteq [0,T]$ is called a set of well-spaced points if for any $t_{1},t_{2}\in \mathcal{W}$, we have $|t_{1}-t_{2}|\geq 1.$
There is a well-spaced set $\mathcal{L}_{\chi}\subseteq \mathcal{U}_{\chi}$ such that
$$\int_{\mathcal{U}_{\chi}}|Q_{v_{0},H}(\chi,1+it)R_{v_{0},H}(\chi,1+it)|^2dt\ll \sum_{t\in \mathcal{L}_{\chi}}|Q_{v_{0},H}(\chi,1+it)R_{v_{0},H}(\chi,1+it)|^2.$$
Let \[\mathcal{U}'=\bigcup_{\chi(mod~k)}\{\chi\}\times \mathcal{L}_{\chi}.\]
Since $\log P_{J}-1\geq \frac{4j^2}{\eta}\log\log Q_{J+1}\geq \frac{2}{\eta}\log\log X$, $P_{J}>(\log X)^{\frac{2}{\eta}}$. By definition of $\mathcal{U}'$, for each $(\chi,t)\in \mathcal{U}'$,  there is a $v\in \mathcal{I}_{J}$ such that $|Q_{v,H_{J}}(\chi,1+it)|>e^{-\alpha_{J}v/H_{J}}$.
By Lemma \ref{Basic large values estimate-prime support},
\[|\mathcal{\mathcal{U}'}|\ll|\mathcal{I}_{J}|(kT)^{2\alpha_{J}+o(1)}(kT)^{\eta}X^{o(1)}\ll T^{\frac{1}{2}-\eta}X^{o(1)}.\]
We now also consider separately the cases
\[\mathcal{U}_{S}:=\{(\chi,t)\in \mathcal{U}':|Q_{v_{0},H}(\chi,1+it)|<(\log X)^{-100}\},\]
\[\mathcal{U}_{L}:=\{(\chi,t)\in \mathcal{U}':|Q_{v_{0},H}(\chi,1+it)|\geq (\log X)^{-100}\}.\]
For $\mathcal{U}_{S}$, applying Lemma \ref{Discrete large sieve for characters},
\begin{align*}
  &\sum_{(\chi,t)\in \mathcal{U}_{S}}|Q_{v_{0},H}(\chi,1+it)R_{v_{0},H}(\chi,1+it)|^2dt \\
 \ll & \frac{1}{(\log X)^{200}}\sum_{(\chi,t)\in \mathcal{U}_{S}}|R_{v_{0},H}(\chi,1+it)|^2\\
  \ll& \frac{1}{(\log X)^{200}}(Xe^{-v/H}+|\mathcal{U}_{S}|(kT)^{\frac{1}{2}})(\log 2kT)\frac{1}{Xe^{-v/H}}\ll\frac{1}{(\log X)^{199}}.
\end{align*}
Now it remains to estimate
\[\sum_{(\chi,t)\in \mathcal{U}_{L}}|Q_{v_{0},H}(\chi,1+it)R_{v_{0},H}(\chi,1+it)|^2.\]
By Lemma \ref{Basic large values estimate-prime support}, we obtain $|\mathcal{U}_{L}|\leq \exp((\log X)^{1/64+o(1)})$.
We now give a pointwise bound to $R_{v_{0},H}(\chi,1+it)$ for $(\chi,t)\in \mathcal{U}'$ as follows.
\begin{equation}\label{the pointwise bound in case3}
\max_{(\chi,t)\in \mathcal{U}_{L}}|R_{v_{0},H}(\chi, 1+it)|\ll \frac{\varphi(k)}{k}(\log X)^{-\frac{1}{16}+o(1)}\frac{\log Q}{\log P}.
\end{equation}
We mainly use Proposition \ref{a Halasz bound} to prove the above inequality. Suppose $\delta(n)=1_{(n,\prod_{P\leq p\leq Q}p)=1}(n)$. Then
\begin{align}\label{the bound for f times a characteristic function}
\begin{aligned}
\mathbb{D}_{k}(f\delta\chi,n\mapsto n^{it};X)^2=&\sum_{\substack{p\leq x\\ p\nmid k}}\frac{1-\text{Re}(p^{-it}\delta(p)\chi(p)f(p))}{p} \geq \mathbb{D}_{k}(f\chi,n\mapsto n^{it};X)^2-\sum_{p=P}^{Q}\frac{1}{p}\\
>& \mathbb{D}_{k}(f\chi,n\mapsto n^{it};X)^2-\frac{1}{64}\log\log X.
\end{aligned}
\end{align}
By bounds (\ref{nonprincipal character estimate}) and (\ref{principal character estimate}),
for $\chi\neq \chi_{1}$ and any $t$ with $|t|\leq 2X$,
\begin{equation}\label{nonprincipal character estimate with delta}
\mathbb{D}_{k}(f\delta\chi,n\mapsto n^{it};X)^2> (1/16)\log\log X
\end{equation}
and for $\chi=\chi_{1}$ and $|t-t_{1}|\geq 1$,
\begin{equation}\label{principal character estimate with delta}
\mathbb{D}_{k}(f\delta\chi_{1},n\mapsto n^{it};X)^2>(1/16)\log\log X.
\end{equation}
Hence applying the above bounds and Proposition \ref{a Halasz bound} with $T_{0}=\frac{1}{2}(\log X)^{\frac{5}{64}}$, we conclude
\begin{equation}\label{Halasz bound with nonprincipal character estimate}
\max_{\substack{\chi(mod~k)\\ \chi\neq \chi_{1}}}\max_{|t|\leq X}|R_{v_{0},H}(\chi, 1+it)|\ll \frac{\varphi(k)}{k}(\log X)^{-\frac{1}{16}+o(1)}\frac{\log Q}{\log P}.
\end{equation}
and for $\chi=\chi_{1}$,
\begin{equation}\label{Halasz bound with principal character estimate}
\max_{|t|\leq X,|t-t_{1}|\geq (\log x)^{\frac{5}{64}}}|R_{v_{0},H}(\chi_{1}, 1+it)|\ll \frac{\varphi(k)}{k}(\log X)^{-\frac{1}{16}+o(1)}\frac{\log Q}{\log P}.
\end{equation}
Note that $\mathcal{U}_{L}\subseteq \mathcal{L}_{2}=\{t\in [0,T]:|t-t_{1}|\geq (\log X)^{\frac{5}{64}}\}$. Hence we obtain formula (\ref{the pointwise bound in case3}). Based on the Hal\'{a}sz bound (\ref{the pointwise bound in case3}) and the condition that $k\leq (\log X)^{1/32}$, it follows, from the similar process in \cite[Section 8.3]{KM} with the Hal\'{a}sz inequality for primes replaced by a hybrid version of it (Lemma \ref{hybrid version of Halasz inequality for primes}), that
\begin{equation}\label{contribution from uchi}
\sum_{\chi(mod~k)}\int_{\mathcal{U}_{\chi}}\Big|\sum_{X\leq m\leq 2X}\frac{f(m)\chi(m)}{m^{1+it}}\Big|^2dt\ll\frac{\varphi(k)}{k}(\varphi(k)T/X+(\varphi(k)/k))(\log X)^{-\frac{1}{64}+o(1)}.
\end{equation}
Combining bounds (\ref{eq111}), (\ref{contribution from E1}), (\ref{contribution from Ej}), (\ref{contribution from uchi}) with formula (\ref{split the integral into several subintegrals}), we obtain
\begin{equation}\label{contribution from the first term}
\sum_{\chi(mod~k)}\int_{\mathcal{L}_{2}}|F(\chi,1+it)|^2dt \ll  \frac{\varphi(k)}{k}(\frac{\varphi(k)T}{X/Q_{1}}+\frac{\varphi(k)}{k})\Big(\frac{(\log Q_{1})^{\frac{1}{3}}}{P_{1}^{\frac{1}{6}-\eta}}+\frac{k}{\varphi(k)}\frac{\log P_{1}}{\log Q_{1}}+\frac{1}{(\log X)^{\frac{1}{65}}}\Big).
\end{equation}
Thanks to equation (\ref{Halasz bound with nonprincipal character estimate}), an argument similar to the proof of (\ref{contribution from the first term}) leads to
\begin{equation}\label{contribution from the second term}
\sum_{\substack{\chi(mod~k)\\ \chi\neq\chi_{1}}}\int_{\mathcal{L}_{1}}|F(\chi,1+it)|^2dt \ll \frac{\varphi(k)}{k}(\frac{\varphi(k)T}{X/Q_{1}}+\frac{\varphi(k)}{k})\Big(\frac{(\log Q_{1})^{\frac{1}{3}}}{P_{1}^{\frac{1}{6}-\eta}}+\frac{k}{\varphi(k)}\frac{\log P_{1}}{\log Q_{1}}+\frac{1}{(\log X)^{\frac{1}{65}}}\Big).
\end{equation}
Now we are just left with estimating
\[\int_{\mathcal{L}_{1}}|F(\chi_{1},1+it)|^2dt.\]
We first assume that
\begin{equation}\label{assumption}
(M_{k}(f\chi_{1};X;2X)+1)\exp(-M_{k}(f\chi_{1};X;2X))>(\log X)^{-\frac{5}{64}}.
\end{equation}
Now we write $\mathcal{L}_{1}=\mathcal{L}_{0,1}\cup \mathcal{L}_{0,2}$ as a disjoint union, where
\[\mathcal{L}_{0,1}=\{t\in \mathcal{L}_{1}: |t-t_{1}|< (M_{k}(f\chi_{1};X;2X)+1)^{-1}\exp(M_{k}(f\chi_{1};X;2X))\},\] and
\[\mathcal{L}_{0,2}=\{t\in \mathcal{L}_{1}: (M_{k}(f\chi_{1};X;2X)+1)^{-1}\exp(M_{k}(f\chi_{1};X;2X))\leq |t-t_{1}|\leq (\log X)^{\frac{5}{64}}\}.\]
For $t\in \mathcal{L}_{0,1}$, by Lemma \ref{the point bound of Dirichlet polynomial} with $T_{0}=(\log X)^{\frac{5}{64}}$, we have for $|t|\leq T\leq X$,
\[F(\chi_{1},1+it)\ll\frac{\varphi(k)}{k}(M_{k}(f\chi_{1};X;2X)+1)\exp(-M_{k}(f\chi_{1};X;2X)).\]
For $t\in \mathcal{L}_{0,2}$, by Lemma \ref{the point bound of Dirichlet polynomial} with $T_{0}=\frac{|t-t_{1}|}{2}$, we have for $|t|\leq T\leq X$,
\[F(\chi_{1},1+it)\ll \frac{\varphi(k)}{k}\frac{1}{|t-t_{1}|},\]
this is because that $(M_{k}(f\chi_{1};X;2X)+1)^{-1}\exp(M_{k}(f\chi_{1};X;2X))\geq 1$ and $(L(f\chi_{1}; X;T_{0})+1)\exp(-L(f\chi_{1}; X; T_{0}))\ll (\log X)^{-\frac{1}{12}+o(1)}$ by equation (\ref{principal character estimate}).
Hence
\begin{equation}\label{contribution from the third term}
\int_{\mathcal{L}_{1}}|F(\chi_{1},1+it)|^2dt\ll \frac{\varphi^2(k)}{k^2}(M_{k}(f\chi_{1};X;2X)+1)\exp(-M_{k}(f\chi_{1};X;2X)).
\end{equation}
Note that $M_{k}(f\chi_{1};X;2X)=M_{k}(f;k;X;2X)$. Therefore, collecting equations (\ref{contribution from the first term}), (\ref{contribution from the second term}) and (\ref{contribution from the third term}), we conclude that
\begin{align}\label{return formula}
\begin{aligned}
 \sum_{\chi(mod~k)}\int_{0}^{T}|F(\chi,1+it)|^2dt \ll&\frac{\varphi(k)}{k}(\frac{\varphi(k)T}{X/Q_{1}}+\frac{\varphi(k)}{k})\Big(\frac{(\log Q_{1})^{\frac{1}{3}}}{P_{1}^{\frac{1}{6}-\eta}}+\frac{k}{\varphi(k)}\frac{\log P_{1}}{\log Q_{1}}+\frac{1}{(\log X)^{\frac{1}{65}}}\Big)\\
  &+\frac{\varphi^2(k)}{k^2}(M_{k}(f;k;X;2X)+1)\exp(-M_{k}(f;k;X;2X)).
\end{aligned}
\end{align}

If condition (\ref{assumption}) does not hold, then $M_{k}(f\chi_{1};X;2X)\geq (5/64-o(1))\log\log X$. So equation (\ref{principal character estimate with delta}) holds for any $|t|\leq 2X$ by equation (\ref{the bound for f times a characteristic function}). Further using (\ref{nonprincipal character estimate with delta}) and Proposition \ref{a Halasz bound} with $T_{0}=(\log X)^{\frac{1}{16}}$,
\[\max_{\chi(mod~k)}\max_{|t|\leq X}|R_{v_{0},H}(\chi, 1+it)|\ll \frac{\varphi(k)}{k}(\log X)^{-\frac{1}{16}+o(1)}\frac{\log Q}{\log P}.\]
By the above pointwise bound, an argument similar to the proof of equation (\ref{contribution from the first term}) leads to
\[\sum_{\chi(mod~q)}\int_{0}^{T}|F(\chi,1+it)|^2dt \ll  \frac{\varphi(k)}{k}(\frac{\varphi(k)T}{X/Q_{1}}+\frac{\varphi(k)}{k})\Big(\frac{(\log Q_{1})^{\frac{1}{3}}}{P_{1}^{\frac{1}{6}-\eta}}+\frac{k}{\varphi(k)}\frac{\log P_{1}}{\log Q_{1}}+\frac{1}{(\log X)^{\frac{1}{65}}}\Big),\]
which implies formula (\ref{return formula}).

Note that $\eta=\frac{1}{150}$. In case $h\leq \exp((\log X)^{1/2})$, we choose $Q_{1}=h$ and $P_{1}=(\log h)^{\frac{40}{\eta}}$; in case $\exp((\log X)^{\frac{1}{2}})\leq h\leq X$, we choose $Q_{1}=\exp((\log X)^{\frac{1}{2}})$, $P_{1}=Q_{1}^{(1/4)(\log h)^{-1/100}}$. Hence from the formula (\ref{return formula}), we obtain the inequality in the statement of this lemma.
\end{proof}

\end{document}